\documentclass{article}
\def\Bbb{\mathbb}
\def\cal{\mathcal}
 \catcode`@=11
\usepackage{amsthm,amssymb,amsfonts,mathrsfs}
\usepackage{amsmath}
\catcode`@=11 \@addtoreset{equation}{section}

\catcode`@=12
\newtheorem{Theorem}{Theorem}[section]
\newtheorem{Proposition}{Proposition}[section]
\newtheorem{Lemma}{Lemma}[section]
\newtheorem{Corollary}{Corollary}[section]
\theoremstyle{definition}

\newtheorem{Definition}{Definition}[section]
\newtheorem{Remark}{Remark}

\newtheorem{Assumptions}{Hypothesis}[section]
\def\R{{\mathbb{R}}}

\def\cH{\mathcal H}

\def\ds{\displaystyle}

\def\fd{\mathfrak{d}}
\def\fg{\mathfrak{g}}
\def\fh{\mathfrak{h}}

\def\vp{\varphi}
\def\ve{\varepsilon}
\def\V1{\mathcal{V}_1}
\def\V2{\mathcal{V}_2}
\def\cK{\mathcal K}
\def\cW{\mathcal W}
\title {Interior degenerate/singular parabolic equations in nondivergence form: well-posedness and Carleman estimates}
\author{{\sc Genni Fragnelli}\\
Dipartimento di Matematica\\ Universit\`{a} di Bari "Aldo Moro"\\
Via
E. Orabona 4\\ 70125 Bari - Italy\\ email: genni.fragnelli@uniba.it}

\date{}
\begin{document}

\maketitle

\vspace{0.3cm}

\centerline{ {\it  }}

\begin{abstract}
We consider non smooth general degenerate/singular parabolic equations in non divergence form with degeneracy and singularity occurring in the interior of the spatial domain, in presence of Dirichlet or Neumann boundary conditions. In particular, we consider well posedness of the problem and then we prove
 Carleman estimates for the associated adjoint problem.
 \end{abstract}


%%%%%%%%%%%%%%%%%%%%%%%%%%%%%%%%%%%%
%%%%%%%%%%%%%%%%%%%%%%%%%%%%%%%%%%%
%%%%%%%%%%%%%%%%%%%%%%%%%%%%%%%%%%%%
%%%%%%%%%%%%%%%%%%%%%%%%%%%%%%%%%%%%
%%%%%%%%%%%%%%%%%%%%%%%%%%%%%%%%%%%%
%%%%%%%%%%%%%%%%%%%%%%%%%%%%%%%%%%%%%%%%
%%%%%%%%%%%%%%%%%%%%%%%%%%%%%%%%%%%%%%%

\section{Introduction}
The present paper is devoted to give a full analysis of the following problem:
\begin{equation}\label{linear}
\begin{cases}
u_t - a(x)u_{xx}  - \displaystyle \frac{\lambda}{b(x)}u=h(t,x) \chi_{\omega}(x), & (t,x) \in Q_T,\\
Bu(0)=Bu(1)=0, &t \in (0,T),\\
u(0,x)=u_0(x), & x \in (0,1),
\end{cases}
\end{equation}
where $Bu(x)= u(t,x)$ or $Bu(x)= u_x(t,x)$ for all $t \in [0,T]$,
$Q_T:=(0,T) \times (0,1)$, $\chi_\omega$ is the characteristic function of a set $\omega \subset (0,1)$, $u_0 \in L^2_{\frac{1}{a}}(0,1)$ and $h \in L^2_{\frac{1}{a}}(Q_T):= L^2(0,T; L^2_{\frac{1}{a}}(0,1))$. Here $L^2_{\frac{1}{a}}(0,1)$ is the Hilbert space
\[L^2_{\frac{1}{a}}(0,1) :=\left\{ u \in L^2(0,1) \
\mid \int_0^1 \frac{u^2}{a} dx <\infty \right\},\]
endowed with the inner product
\[
\langle u,v\rangle_{L^2_{\frac{1}{a}}(0,1)}^2:= \int_0^1
\frac{uv}{a} dx, \quad \mbox{ for every }u,v\in L^2_{\frac{1}{a}}(0,1),
\]
which induces the obvious associated norm.

Moreover, we assume that the constant $\lambda$ satisfies suitable assumptions
described below and the functions $a$ and
$b$, that can be {\it non smooth}, degenerate at the same interior point $x_0 \in
(0,1)$ that can belong to the control set $\omega$. The fact that both $a$ and $b$ degenerate at $x_0$ is just for the sake of simplicity and shortness: all the stated results are still valid if they degenerate at different points.
We shall admit different types of degeneracy for $a$ and $b$. In
particular, we make the following assumptions:
\begin{Assumptions}\label{Ass0}
{\bf Double weakly degenerate case (WWD):} there exists $x_0\in
(0,1)$ such that $a(x_0)=b(x_0) =0$, $a, b>0$ on $[0, 1]\setminus
\{x_0\}$, $a, b\in W^{1,1}(0,1)$ and there exist $K_1, K_2 \in
(0,1)$ such that $(x-x_0)a' \le K_1 a$ and $(x-x_0)b' \le K_2 b$
a.e. in $[0,1]$.
\end{Assumptions}

\begin{Assumptions}\label{Ass01}
{\bf Double strongly degenerate case (SSD):} there exists $x_0 \in
(0,1)$ such that $a(x_0)=b(x_0)=0$, $a, b>0$ on $[0, 1]\setminus
\{x_0\}$, $a, b\in W^{1, \infty}(0,1)$ and there exist $K_1, K_2 \in
[1,2)$ such that $(x-x_0)a' \le K_1 a$ and $(x-x_0)b' \le K_2 b$
a.e. in $[0,1]$.
\end{Assumptions}

\begin{Assumptions}\label{Ass0_1}
{\bf Weakly strongly degenerate case (WSD):} there exists $x_0 \in
(0,1)$ such that $a(x_0)=b(x_0)=0$, $a, b>0$ on $[0, 1]\setminus
\{x_0\}$, $a\in W^{1,1}(0,1)$, $ b\in W^{1, \infty}(0,1)$ and there
exist $K_1\in (0,1)$, $K_2 \in [1,2)$ such that $(x-x_0)a' \le K_1
a$ and $(x-x_0)b' \le K_2 b$ a.e. in $[0,1]$.
\end{Assumptions}

\begin{Assumptions}\label{Ass01_1}
{\bf Strongly weakly degenerate case (SWD):} there exists $x_0 \in
(0,1)$ such that $a(x_0)=b(x_0)=0$, $a, b>0$ on $[0, 1]\setminus
\{x_0\}$, $a\in W^{1, \infty}(0,1)$, $b\in W^{1,1}(0,1)$, and there
exist $K_1\in [1,2)$, $K_2 \in (0,1)$ such that $(x-x_0)a' \le K_1
a$ and $(x-x_0)b' \le K_2 b$ a.e. in $[0,1]$.
\end{Assumptions}
Typical examples for the previous degeneracies and singularities are
$a(x)=|x- x_0|^{K_1}$ and $b(x)= |x-x_0|^{K_2}$, with $ 0<K_1,
K_2<2$.

In the last recent years an increasing interest has been devoted to \eqref{linear} in the case when $\lambda=0$.
For example, we recall the works \cite{acf}, \cite{bcg}, \cite{bfm}, \cite{bfm1}, \cite{cfr}-\cite{cmv1}, \cite{f}-\cite{fggr}, where the authors focus their attention mainly on  well posedness and on global null controllability for \eqref{linear}, also via Carleman estimates (for the nonlinear case see also \cite{fl}). We recall that \eqref{linear} is said globally null controllable if for
every $u_0 \in L^2_{\frac{1}{a}}(0,1)$ there exists $h \in L^2_{\frac{1}{a}}(Q_T)$ such that
the solution $u$ of \eqref{linear}
satisfies
$
u(T,x)= 0 \ \text{ for every  } \  x \in [0, 1]
$
and
$
\|h\|^2_{L^2_{\frac{1}{a}}(Q_T)} \le C \|u_0\|^2_{L^2_{\frac{1}{a}}(0,1)}
$
for some universal positive constant $C$.

If $\lambda \neq0$, the
first results in this direction are obtained in \cite{vz2} for the
heat operator with singular potentials
 \begin{equation}\label{4}
  u_t -
u_{xx}-\lambda\frac{1}{x^{K_2}}u, \quad (t,x) \in Q_T,
\end{equation}
and Dirichlet boundary conditions.
The case ${K_2}=2$ is the critical one and it is the case of the so-called inverse square
potential that arises for example in quantum mechanics (see, e.g., \cite{bg}, \cite{d}) or in combustion problems
(see, e.g., \cite{bv}, \cite{dglv}, \cite{gv}).
This potential is known to generate interesting phenomena: in \cite{bg} and in \cite{bg1} it is proved, for example, that if ${K_2} < 2$ then global
positive solutions exist for any value of $\lambda$,
whereas,  if ${K_2} > 2$ then instantaneous and complete blow-up occurs for any value of
$\lambda$. Finally, when ${K_2} = 2$, the value of the parameter determines the behavior of
the equation:  if $\lambda \le
\displaystyle\frac{1}{4}$ (which is the optimal constant of the
Hardy inequality) then global positive solutions exist, whereas, if $\lambda > \displaystyle\frac{1}{4}$ then instantaneous and complete blow-up occurs.

Moreover, in \cite{e}, \cite{fs}, \cite{fs1}, \cite{v1} and \cite{vz2}, great attention is given to null controllability in the case $\lambda \neq 0$. Indeed, in \cite{vz2}, new Carleman estimates (and consequently null
controllability properties) were established for \eqref{4}
 under the
condition $\lambda \le \displaystyle\frac{1}{4}$. On the contrary,
if $\lambda >\displaystyle\frac{1}{4}$, in \cite{e}, it was proved
that null controllability fails.

Recently, in \cite{v1}, J. Vancostenoble studies the operator that
couples a degenerate diffusion coefficient with a singular
potential. In particular, under suitable conditions on ${K_1}$,
${K_2}$ and $\lambda$, the author established Carleman estimates for
the operator
\[
u_t - (x^{K_1} u_x)_x-\lambda\frac{1}{x^{K_2}}u, \quad (t,x) \in
Q_T,
\]
unifying the results of \cite{cmv} and \cite{vz2} in the purely
degenerate operator and in the purely singular one, respectively.
This result was then extended in \cite{fs} and in \cite{fs1} to the
operators
\begin{equation}\label{ultimo?}
u_t - (a(x) u_x)_x-\lambda\frac{1}{x^{K_2}}u, \quad (t,x) \in Q_T,
\end{equation}
under different assumptions on $a$ and ${K_2}$. Here, as before, the
function $a$ degenerates at the boundary of the space domain and Dirichlet boundary conditions are in force.

However, all the previous papers deal with a degenerate/singular
operator with degeneracy or singularity at the boundary of the
domain. For example, in \eqref{ultimo?} as $a$, one can consider the double power
function
\[a(x)= x^k(1-x)^\alpha, \quad x \,\in \,[0,1],\]
where $k$ and $\alpha$ are positive constants. To the best of our knowledge, \cite{bfm}, \cite{bfm1}, \cite{fm}, \cite{fm1} and \cite{fggr} are the first
papers deal with well posedness and  Carleman estimates  (and, consequently, null
controllability) for operators (in divergence and in non divergence
form with Dirichlet or Neumann boundary conditions) with purely degeneracy (i.e. $\lambda=0$) at the interior of the
space domain.
In particular, \cite{fm1} is the first paper that deals with a {\it non smooth} degenerate function $a$.

Recently, in \cite{fm2} the authors treat for the first time well posedness and  null controllability for operator with Dirichlet boundary conditions in {\it divergence form } with a degeneracy and a singularity (i.e. $\lambda \neq 0$) both that occurring in the interior of the domain (we refer to \cite{fm2} for other references on this subject). We underline the fact that in the present paper we cannot use the results of \cite{fm2}, since the equation in non divergence form {\it cannot} be recast, in general, from the equation in
divergence form: for example, if $\lambda =0$, it was proved in \cite{fggr} that the simple equation
\[
u_t=a(x)u_{xx}
\]
can be written in divergence form as
\[
u_t=(au_x)_x-a'u_x,
\]
only if $a'$ does exist; in addition, even if $a'$ exists, considering  well-posedness for the last equation, additional conditions are
necessary: for instance, for the prototype $a(x)= |x-x_0|^{K_1}$,
well-posedness is guaranteed if ${K_1} \ge 2$ (see \cite{fggr}). However, in
\cite{fm1} the authors prove that if $a(x)=|x-x_0|^{K_1}$ global null
controllability fails exactly when ${K_1} \ge 2$. Thus, it is important to prove directly that, under suitable
conditions for which well-posedness holds, the problem in non divergence form is still
globally null controllable.
\medskip

For this reason, the object of this paper is twofold:  first we analyze well-posedness of
\eqref{linear} for a \textit{ge\-ne\-ral}
degenerate diffusion coefficient and a \textit{general} singular
potential, with degeneracy and singularity at the \textit{interior}
of the space domain; second, under suitable conditions on all
the parameters of \eqref{linear}, we prove related global Carleman estimates. 
Finally, as a consequence of Carleman estimates, using a reflection procedure, we prove an obser\-va\-bility inequality: there exists a positive constant $C_T$ such that every solution
$v$ of the adjoint problem
\[
\begin{cases}
v_t +av_{xx} +\displaystyle \frac{ \lambda}{b(x)}v= 0, &(t,x) \in
Q_T,
\\
Bv(0)=Bv(1) =0, & t \in (0,T),
\\
v(T,x)= v_T(x)\in L^2_{\frac{1}{a}}(0,1),
\end{cases}
\]
satisfies, under suitable assumptions, 
 \begin{equation}\label{obs}
\|v(0)\|^2_{L^2_{\frac{1}{a}}(0,1)} \le C_T\|v\chi_\omega\|^2_{L^2_{\frac{1}{a}}(Q_T)}.
\end{equation}
As an immediate consequence, one can prove, using a standard
technique (e.g., see \cite[Section 7.4]{LRL}), 
null controllability for the linear degenerate/singular problem \eqref{linear}.

Clearly, this result generalizes the result obtained in
\cite{bfm1}: in fact, if we consider Neumann boundary conditions and if $\lambda =0$ (that is, if we consider the
purely degenerate case), we obtain exactly the result of \cite{bfm1} in the case of a problem in non divergence form.

Finally, we remark that also in the case of degenerate and singular problems  a key step in
the proof of Carleman estimates is not only the correct choice
of the weight functions, but also some  special inequalities that we will
show later, together with Hardy--Poincar\'{e} inequalities (see Subsections \ref{HPDBC} and \ref{HPNBC}).
\medskip

The paper is organized in the following way: in Section \ref{secPR}, which is divided into two subsections,  we give some preliminary results, such as Hardy--Poincar\'e inequalities, that will be useful for the rest of the paper. In Section \ref{sec3}
we study well posedness of the problem applying the previous inequalities. In Section \ref{sec4}, we prove Carleman estimates and we use them, together with a Caccioppoli type
inequality, to prove
observability inequalities in Section \ref{sec5}.

A final comment on the notation: by $C$ we shall denote 
universal positive constants, which are allowed to vary from line to
line.  Moreover, in the rest of the paper we will write, for shortness, {\bf (Dbc)} or {\bf (Nbc)} in place of Dirichlet boundary conditions or Neumann ones, respectively.

\section{Preliminary results}\label{secPR}
In this part of the paper we give different weighted  Hardy--Poincar\'e inequalities that will be very important for the rest of the paper. 
In particular, we divide this section into two subsections. In the first one we give Hardy--Poincar\'e inequalities in the case of Dirichlet boundary conditions; in the last one we prove them in the case of Neumann ones. In order to deal with these  inequalities we  consider different classes of weighted Hilbert spaces,
which are suitable to study the four different situations given in the Introduction. We remark that we shall use the standard notation $H$ for
spaces with degenerate weights and (Dbc) and the calligraphic
notation ${\cal H}$ for spaces with degenerate weights and (Nbc). Thus, we introduce
\[
\cK_a:= \begin{cases}
H^1_{\frac{1}{a}}(0,1) := L^2_{\frac{1}{a}}(0,1)\cap H^1_0(0,1), & \text{if (Dbc) hold},\\
\cH^1_{\frac{1}{a}}(0,1) :=L^2_{\frac{1}{a}}(0,1)\cap H^1(0,1), & \text{if (Nbc) are in force},
\end{cases}
\]
and
\[\cK_{a,b} :=\left \{u \in \cK_a: \frac{u}{\sqrt{ab}}\in L^2(0,1)\right\}\]
with  the inner products
\[
\langle u, v \rangle_{\cK_a}=\int_0^1 \frac{uv}{a}dx + \int_0^1 u'v' dx, 
\]
and
\[
\langle u,v\rangle_{\cK_{a,b}} =  \int_0^1 \frac{uv}{a}dx + \int_0^1 u'v' dx + \int_0^1 \frac{uv}{ab}dx, 
\]
respectively.

\vspace{0.5cm}
Moreover, we will use the following results several times; we state the first lemma for $a$, but an analogous one holds for  $b$ replacing $K_1$ with $K_2$:
\begin{Lemma}[Lemma 2.1, \cite{fm}]\label{Lemma 2.1}
Assume that there exists $x_0 \in (0,1)$ such that $a(x_0)=0$, $a>0$ on $[0, 1]\setminus
\{x_0\}$, and
\begin{itemize}
\item $a\in W^{1, 1}(0,1)$ and there
exist $K_1\in (0,1)$ such that $(x-x_0)a' \le K_1
a$ a.e. in $[0,1]$, or
\item $a\in W^{1, \infty}(0,1)$ and there
exist $K_1\in [1,2)$ such that $(x-x_0)a' \le K_1
a$ a.e. in $[0,1]$.
\end{itemize}
\begin{enumerate}
\item Then for all $\gamma \ge K_1$ the map
$$
\begin{aligned}
& x \mapsto \dfrac{|x-x_0|^\gamma}{a} \mbox { is non increasing on
the left of } x=x_0 \\
& \mbox{and non decreasing on the right of }
x=x_0,\\
&\mbox{ so that }\lim_{x\to x_0}\dfrac{|x-x_0|^\gamma}{a}=0 \mbox{
for all }\gamma>K_1.
\end{aligned}
$$
\item If $K_1<1$, then
    $\displaystyle\frac{1}{a} \in L^{1}(0,1)$.\\
\item If $K _1\in[1,2)$, then $\displaystyle \frac{1}{\sqrt{a}} \in
    L^{1}(0,1)$ and $\displaystyle \frac{1}{a}\not \in L^1(0,1)$.
\end{enumerate}
\end{Lemma}
For the next result we make the  following assumption:
\begin{Assumptions}\label{Assnew}
The  functions $a, b$ are such that
\begin{enumerate}
\item
$a, b \in W^{1,\infty}(0,1)$,
or
\item
$
a, b \in W^{1,1}(0,1)$ and there exist $K_1, K_2, c_1, c_2 >0$ such that $K_1+ K_2 \ge 1$ and
\begin{equation}\label{iponew}
  \ds |x-x_0|^{K_1}\ge c_1 a\;\text{ and }\; \ds |x-x_0|^{K_2} \ge c_2b\; \text{  for all  } x \in  [0,1].
\end{equation}
\end{enumerate}
\end{Assumptions}
Observe that the last assumption is not restrictive. Indeed, if we consider the prototype functions $a(x) = |x-x_0|^{K_1}$ and $b(x) = |x-x_0|^{K_2}$, with $K_1+K_2\ge1$, the last part of Hypothesis $\ref{Assnew}.2$ is clearly  satisfied with $c_1 =c_2=1$.
\begin{Lemma}\label{leso} Assume that Hypothesis $\ref{Assnew}$ holds. Then
\begin{enumerate}
\item $\displaystyle \frac{1}{ab} \not \in L^1(0,1)$;
\item $u(x_0)=0$ for every $u\in \cK_{a,b}(0,1)$.
\end{enumerate}
\end{Lemma}
\begin{proof}
{\bf 1.} First of all assume
that  Hypothesis $\ref{Assnew}.1$ is satisfied. Then the assumptions on $a$ and $b$ imply $(ab)(x) = \displaystyle \int_{x_0}^x(ab)'(s)ds.$ 
Thus there exists a positive constant $C$ such that, if Hypothesis $\ref{Assnew}.1$ is satisfied, then
\[
(ab)(x)=|(ab)(x)| \le C|x-x_0|.
\]
Hence, for all $x\neq x_0$ and for a suitable constant $C>0$, $\displaystyle
\frac{1}{(ab)(x)} \ge C \frac{1}{|x-x_0|} \not \in L^1(0,1)$.
\\
Assume now that Hypothesis $\ref{Assnew}.$2 is satisfied. Then
\[
\frac{1}{ab} \ge \frac{c_1c_2}{|x-x_0|^{K_1+ K_2}} \not \in L^1(0,1),
\]
being $K_1+ K_2 \ge 1$.
\\
{\bf 2.}
Since $u\in W^{1,1}(0,1)$, there exists $\lim_{x\to x_0}u(x)=L\in \R$. If $L\neq 0$, then $\displaystyle |u(x)|\geq \frac{L}{2}$ in a neighborhood of $x_0$, that is
\[
\frac{|u(x)|^2}{ab}\geq \frac{L^2}{4ab}\not \in L^1(0,1)
\]
by the first point, and thus $L=0$.
\end{proof}
We also need the following result, whose proof, with the aid of Lemma \ref{leso}, is a simple adaptation of the one given in \cite[Lemma 3.2]{fggr}.
\begin{Lemma}\label{densita}
Assume that Hypothesis $\ref{Assnew}$ is satisfied. Then
\[
H^1_c(0,1):=\Big\{u\in H_0^1(0,1)\mbox{ \rm such that supp}\,u\subset (0,1)\setminus\{x_0\}\Big\}
\] 
is dense in $\cK_{a,b} :=\left \{u \in H^1_{\frac{1}{a}}(0,1):\ds \frac{u}{\sqrt{ab}}\in L^2(0,1)\right\}$.
\end{Lemma}

 \subsection {Hardy--Poincar\'e inequalities in the case of (Dbc)}\label{HPDBC}
The first inequality is proved in
\cite[Proposition 2.6]{fm} (we refer also to \cite[Proposition 1.1]{fm1} for some comments):
\begin{Proposition}\label{HP}
Assume that $p \in C([0,1])$, $p>0$ on $[0,1]\setminus \{x_0\}$,
$p(x_0)=0$ and there exists $q >1$ such that the function
\begin{equation}\label{p}
\begin{aligned}
x \mapsto \dfrac{p(x)}{|x-x_0|^{q}} &\mbox { is
non increasing on the left of } x=x_0 \\
& \mbox{ and non
decreasing on the right of } x=x_0.
\end{aligned}
\end{equation}
\noindent Then, there exists a constant $C_{HP}>0$ such that for any
function $w$, locally absolutely continuous on $[0,x_0)\cup (x_0,1]$, satisfying
$$
w(0)=w(1)=0 \,\, \mbox{and } \int_0^1 p(x)|w^{\prime}(x)|^2 \,dx <
+\infty \,,
$$ the following inequality holds:
\begin{equation}\label{hardy1}
\int_0^1 \dfrac{p(x)}{(x-x_0)^2}w^2(x)\, dx \leq C_{HP}\, \int_0^1
p(x) |w^{\prime}(x)|^2 \,dx.
\end{equation}
\end{Proposition}

\vspace{0.5cm}

Using the weighted spaces introduced before we can prove the next Hardy--Poincar\'e inequalities.
First, we make the following assumption:
\begin{Assumptions}\label{Ass03_new}
\begin{enumerate}
\item
Hypothesis $\ref{Ass0}$ holds with $K_1+K_2 <1$, or
\item Hypothesis $\ref{Ass0}$ holds with $1\le  K_1+K_2 \le 2$ and 
\begin{equation}\label{iponewnew}
\exists \; c_1, c_2 >0 \text{ such that }  \ds |x-x_0|^{K_1}\ge c_1 a\; \text{ and }\; \ds |x-x_0|^{K_2} \ge c_2b\; \text{  for all  } x \in  [0,1],
\end{equation}
or
\item Hypothesis $\ref{Ass0_1}$ or $\ref{Ass01_1}$  holds with $K_1+K_2\leq 2$ and \eqref{iponewnew}, or 
\item Hypothesis $\ref{Ass01}$ holds with $K_1=K_2=1$.
\end{enumerate}
\end{Assumptions}
\begin{Lemma}\label{L2} Assume that Hypothesis $\ref{Ass03_new}.1$ holds.
Then  there exists a constant $\overline C_{HP} >0$ such that
\begin{equation}\label{1}
\int_0^1 \frac{u^2}{ab} \le  \overline C_{HP}  \int_0^1 (u')^2dx
\end{equation}
 for every $u
\in \cK_a$.
\end{Lemma}
\begin{proof}
Let $u \in \cK_a$ (recall that in this case $(\cK_a=H^1_{\frac{1}{a}}(0,1))$ and define
$p(x):= \displaystyle
\frac{(x-x_0)^2}{ab}$. Using Lemma \ref{Lemma 2.1} and the assumtpion $K_1+ K_2 <1$, one has that
 the function $\displaystyle\frac{p(x)}{|x-x_0|^q}$,
where $\displaystyle q: =2-(K_1+K_2) >1$, is
non increasing on the left of $x=x_0$ and non decreasing on the right
of $x=x_0$. Thus, Proposition \ref{HP} implies,
\[
\int_0^1 \frac{u^2}{ab}\,dx = \int_0^1 p\frac{u^2}{(x-x_0)^2}\,dx \le C_{HP} \int_0^1
p(u')^2dx \le \beta C_{HP}\int_0^1 (u')^2 dx,
\]
for a positive constant $C_{HP}$, being
\begin{equation}\label{beta}
\beta:=
\max\left\{\displaystyle\frac{x_0^2}{(ab)(0)},\displaystyle\frac{(1-x_0)^2}{(ab)(1)}\right\}.
\end{equation}
Hence \eqref{1} is satisfied with $\overline C_{HP}= \beta C_{HP}$.
\end{proof}
\begin{Lemma}\label{L2''}
Assume that one among Hypothesis $\ref{Ass03_new}.2$, $\ref{Ass03_new}.3$ or $\ref{Ass03_new}.4$ is satisfied. Then there exists a constant $\overline C_{HP}>0$ such that
\eqref{1} holds
 for every $u
\in \cK_{a, b}$.
\end{Lemma}
\begin{proof}By Lemma \ref{leso} we know that, taken $u\in \cK_{a,b}$,  $u(x_0) =0$. Fix $\ve\in \big(0, \min\{x_0,1-x_0\}\big)$ and write
\[
\int_0^1\frac{u^2}{ab}dx= \left(\int_0^{x_0-\ve}+\int_{x_0-\ve}^{x_0}+\int_{x_0}^{x_0+\ve}+\int_{x_0+\ve}^1\right)\frac{u^2}{ab}dx.
\]
Now, by the Poincar\'e inequality applied to functions in $[0,x_0-\ve]$ vanishing at $0$, we get
\begin{equation}\label{primopezzoD}
\begin{aligned}
\int_0^{x_0-\ve}\frac{u^2}{ab}dx& \leq \frac{1}{\ds\min_{[0,x_0-\ve]}ab}\int_0^{x_0-\ve}u^2dx\leq  \frac{1}{\ds\min_{[0,x_0-\ve]}ab}\int_0^{x_0-\ve}(u')^2dx\le  C\int_0^1(u')^2dx,
\end{aligned}
\end{equation}
for some $C>0$ independent of $u$. A similar estimate holds for $\ds\int_{x_0+\ve}^1\frac{u^2}{ab}dx$.
\\
Moreover, by Lemma \ref{Lemma 2.1}, there exists $C=C(a,b)>0$ such that
\begin{equation}\label{primopezzo2D}
\int_{x_0-\ve}^{x_0}\frac{u^2}{ab}dx\le C\int_{x_0-\ve}^{x_0}\frac{u^2}{|x-x_0|^{K_1+K_2}}dx\leq C\int_{x_0-\ve}^{x_0}\frac{u^2}{|x-x_0|^2}dx,
\end{equation}
being $K_1+ K_2 \le 2$.
 Since $u(x_0)=0$, the classical Hardy--Poincar\'e inequality  implies that
\begin{equation}\label{primopezzo3D}
\int_{x_0-\ve}^{x_0}\frac{u^2}{ab}dx\le  C\int_{x_0-\ve}^{x_0}(u')^2dx,
\end{equation}
for a suitable constant $C$.
By \eqref{primopezzoD} and \eqref{primopezzo3D}, and operating in a similar way in $[x_0,1]$, the claim follows.
\end{proof} 

\vspace{0.5cm}

Observe that the previous estimates give Hardy--Poincar\'e inequalities in all situations, namely the {\em (WWD)}, {\em (SSD)}, {\em
(WSD)} and {\em (SWD)}. However, Lemma \ref{L2''} allows us to consider for the (SSD) case only the situation when $K_1$ and $K_2$ are both $1$.

\subsection{Hardy--Poincar\'e inequalities in the case of (Nbc)}\label{HPNBC}

In this subsection we give the analogous Hardy--Poincar\'e inequalities stated before for the case of Dirichlet boundary conditions.

In particular, the following inequality is the analogous of Proposition \ref{HP} in the Neumann case:
\begin{Proposition}\label{HPN}
Assume that $p \in C([0,1])$, $p>0$ on $[0,1]\setminus \{x_0\}$,
$p(x_0)=0$ and there exists $q >1$ such that
\eqref{p} holds.
 Then, there exists a constant $C_{HP}>0$ such that for any
function $w$, locally absolutely continuous on $[0,x_0)\cup (x_0,1]$, satisfying
$$
w'(0)=w'(1)=0 \,\, \mbox{and } \int_0^1 |w^{\prime}(x)|^2 \,dx <
+\infty \,,
$$ the following inequality holds:
\begin{equation}\label{hardy2}
\int_0^1 \dfrac{p(x)}{(x-x_0)^2}w^2(x)\, dx \leq C_{HP}\, \int_0^1
p(x) |w^{\prime}(x)|^2 \,dx + 2\Xi\left[ w^2(1) \frac{p(1)}{(1-x_0)^q} + w^2(0)\frac{p(0)}{x_0^q}\right].
\end{equation}
Here
\[
\Xi:=\max\left\{\frac{(1-x_0)^{q-1}}{q-1}, \frac{1}{q-1} \right\}.
\]
\end{Proposition}
\begin{proof}
Fix any $\beta \in (1, q)$ and $\ve>0$ small. Then, since \eqref{p} holds:
\[
\begin{aligned}
&\int_{x_0+\ve}^1 \frac {p(x)}{(x-x_0)^2}w^2(x)\, dx = \int_{x_0+\ve}^1 \frac {p(x)}{(x-x_0)^2}\left[ w(1)-\int_x^1 w'(y) dy\right]^2 dx\\
&\le 2 w^2(1)\int_{x_0+\ve}^1 \frac {p(x)}{(x-x_0)^2}dx + 2 \int_{x_0+\ve}^1\frac {p(x)}{(x-x_0)^2} \left(\int_x^1 w'(y) dy\right)^2 dx\\
&= 2 w^2(1)\int_{x_0+\ve}^1 \frac {p(x)}{(x-x_0)^q}(x-x_0)^{q-2}dx + 2 \int_{x_0+\ve}^1\frac {p(x)}{(x-x_0)^2} \left(\int_x^1 w'(y) dy\right)^2 dx\\
& \le 2 w^2(1)\frac {p(1)}{(1-x_0)^q}\int_{x_0+\ve}^1 (x-x_0)^{q-2}dx + 2 \int_{x_0+\ve}^1\frac {p(x)}{(x-x_0)^2} \left(\int_x^1 w'(y) dy\right)^2 dx\\
& \le 2 w^2(1)\frac {p(1)}{(1-x_0)^q}\frac{(1-x_0)^{q-1}}{q-1}+ 2 \int_{x_0+\ve}^1\frac {p(x)}{(x-x_0)^2} \left(\int_x^1 w'(y) dy\right)^2 dx.
\end{aligned}
\]
Moreover, proceeding as in \cite[Proposition 2.6]{fm}, one can prove that
\begin{equation}\label{HP1N}
\begin{aligned}
\int_{x_0+\ve}^1\!\! \frac {p(x)}{(x-x_0)^2} \left(\int_x^1 w'(y) dy\right)^2 dx &=  \int_{x_0+\ve}^1\!\!\frac {p(x)}{(x-x_0)^2}  \left(\int_x^1 ((y-x_0)^{\frac{\beta}{2}}w'(y))(y-x_0)^{-\frac{\beta}{2}} dy\right)^2\!\!\! dx\\
&\le  \dfrac{1}{(\beta-1)(q-\beta)}\int_{x_0+\ve}^1
p(y)|w^{\prime}(y)|^2 dy.
\end{aligned}
\end{equation}
Hence
\[
\int_{x_0+\ve}^1\!\! \frac {p(x)}{(x-x_0)^2}w^2(x)\, dx\le 2 w^2(1)\frac {p(1)}{(1-x_0)^q}\frac{(1-x_0)^{q-1}}{q-1}+ \dfrac{2}{(\beta-1)(q-\beta)}\int_{x_0+\ve}^1\!\!
p(y)|w^{\prime}(y)|^2 dy.
\]
Analogously, one has

\begin{equation}\label{HP2N}
\begin{aligned}
&\int_0^{x_0-\ve} \frac {p(x)}{(x_0-x)^2}w^2(x)\, dx \le  2 w^2(0)\frac {p(0)}{x_0^q}\frac{1}{q-1} +
\dfrac{2}{(\beta-1)(q-\beta)}\int_0^{x_0-\ve}p(y)|w^{\prime}(y)|^2
dy.
\end{aligned}
\end{equation}
Passing to the limit as $\ve\to 0$ and combining \eqref{HP1N} and
\eqref{HP2N}, the conclusion follows.
\end{proof}
As a consequence, one has the next result:
\begin{Corollary}\label{HPN1}
Assume that $p \in C([0,1])$, $p>0$ on $[0,1]\setminus \{x_0\}$,
$p(x_0)=0$ and there exists $q >1$ such that
\eqref{p} holds.
 Then,
 \begin{enumerate}
 \item there exists a positive constant $C_{HP,1}$ such that   for any
function $w \in H^1(0,1)$ satisfying
$
w'(0)=w'(1)=0$, the following inequality holds:
\begin{equation}\label{HPN2}
 \int_0^1 \dfrac{p(x)}{(x-x_0)^2}w^2(x)\, dx \leq C_{HP,1}\, \|w\|^2_{H^1(0,1)};\end{equation}
\item
  for all $y_0
\in [0,1]$, there exists $C_{HP,2}>0$ such that for any
function $w \in H^1(0,1)$ satisfying
$
w'(0)=w'(1)=0$, the following inequality holds:
\begin{equation}\label{hardy3}
\int_0^1 \dfrac{p(x)}{(x-x_0)^2}w^2(x)\, dx \leq C_{HP,2} \left(\int_0^1 (w')^2(y) dy +
w^2(y_0) \right).
\end{equation}
\end{enumerate}
\end{Corollary}
\begin{proof} {\it 1.}:
Since $H^1(0,1) $ is continuously embedded in
$L^\infty(0,1)$, one has that for all $w \in H^1(0,1)$
\[
|w(y_0)| \le \|w\|_{L^\infty(0,1)}\le C \|w\|_{H^1(0,1)},\; \forall \; y_0 \in [0,1],
\]
for a positive constant $C$.
In particular, $w^2(0)$ and $w^2(1)$ can be estimated by $C
\|w\|_{H^1(0,1)}^2$. Thus, by Proposition \ref{HPN}, \eqref{HPN2} follows immediately.
\\
{\it 2.}: Fix now $y_0\in [0,1]$.
Since the standard $H^1$- norm is equivalent to the norm
\[
\|w\|_{y_0}:= \|w'\|_{L^2(0,1)} + |w(y_0)|,
\]
\eqref{hardy3} follows immediately by \eqref{HPN2}.
\end{proof}

We will proceed with some estimates similar to the ones given in Lemmas \ref{L2} and \ref{L2''}.

\begin{Lemma}\label{L2'}
Assume that Hypothesis $\ref{Ass03_new}.1$ holds. Then there exists a constant $\overline C_{HP}>0$ such that
\begin{equation}\label{1nondiv}
\int_0^1 \frac{u^2}{ab} \le \overline C_{HP}\|u\|^2_{H^1(0,1)}
\end{equation}
for all $u
\in \cK_a(0,1)$ with $u'(0)=u'(1)=0$.
Moreover, if $u(x_0)=0$, then 
\begin{equation}\label{2nondiv}
\int_0^1\frac{u^2}{ab}dx\le \overline C_{HP} \int_0^1 (u')^2(x) dx.
\end{equation}
\end{Lemma}
\begin{proof} Let 
$u \in \cK_a$ (recall that in this case $(\cK_a=\cH^1_{\frac{1}{a}}(0,1))$ and define
$p(x):= \displaystyle
\frac{(x-x_0)^2}{ab}$. As in Lemma \ref{L2}, one can prove that the function $p$ satisfies the assumptions of Corollary \ref{HPN1}, thus, applying \eqref{HPN2}, one has
\[
\int_0^1 \frac{u^2}{ab}\,dx = \int_0^1 p\frac{u^2}{(x-x_0)^2}\,dx \le C_{HP,1}\, \|u\|^2_{H^1(0,1)}.
\]
Hence, \eqref{1nondiv} holds with $\overline C_{HP} = C_{HP,1}$.
Moreover, if $u(x_0)=0$, we can apply Corollary \ref{HPN1}.2, obtaining
\[
\int_0^1\frac{u^2}{ab} dx  = \int_0^1 p\frac{u^2}{(x-x_0)^2} \,dx \le C_{HP,2} \left(\int_0^1 (u')^2(x) dx +
u^2(x_0) \right)= C_{HP,2}\int_0^1 (u')^2(x) dx.
\]
In this case $\overline C_{HP}= C_{HP,2}$.
\end{proof}

\begin{Lemma}\label{L2'''}
Assume that one among Hypothesis $\ref{Ass03_new}.2$, $\ref{Ass03_new}.3$ or $\ref{Ass03_new}.4$ is satisfied. Then there exists a constant $\overline C_{HP}>0$ such that
\eqref{2nondiv} holds for every $u \in \cK_{a,b}$  with $u'(0)=u'(1)=0$.
\end{Lemma}
\begin{proof}
By Lemma \ref{leso} we know that, taken $u \in \cK_{a,b}$, $u(x_0) =0$. As in Lemma \ref{L2''}, fix $\ve\in \big(0, \min\{x_0,1-x_0\}\big)$ and write
\[
\int_0^1\frac{u^2}{ab}dx= \left(\int_0^{x_0-\ve}+\int_{x_0-\ve}^{x_0}+\int_{x_0}^{x_0+\ve}+\int_{x_0+\ve}^1\right)\frac{u^2}{ab}dx.
\]
Now, 
\begin{equation}\label{primopezzo}
\int_0^{x_0-\ve}\frac{u^2}{ab}dx \leq \frac{1}{\ds\min_{[0,x_0-\ve]}ab}\int_0^{x_0-\ve}u^2dx.
\end{equation}
A similar estimate holds for $\ds\int_{x_0+\ve}^1\frac{u^2}{ab}dx$.
Moreover, by Lemma \ref{Lemma 2.1}, there exists $C=C(a,b)>0$ such that
\begin{equation}\label{primopezzo2}
\int_{x_0-\ve}^{x_0}\frac{u^2}{ab}dx\le C \int_{x_0-\ve}^{x_0}\frac {u^2}{|x-x_0|^{K_1+K_2}}dx\leq C\int_{x_0-\ve}^{x_0}\frac{u^2}{|x-x_0|^2}dx.
\end{equation}
Being $u(x_0)=0$, the classical Hardy--Poincar\'e inequality implies
\begin{equation}\label{primopezzo3}
\int_{x_0-\ve}^{x_0}\frac{u^2}{ab}dx\le C\int_{x_0-\ve}^{x_0}\frac{u^2}{|x-x_0|^2}dx \le C\int_{x_0-\ve}^{x_0} (u')^2dx,
\end{equation}
for a positive constant $C$.
An analogous estimate holds also in $[x_0, x_0+\ve]$. Hence
\[
\begin{aligned}
\int_0^1\frac{u^2}{ab}dx&\le \frac{2}{\ds\min_{[0,x_0-\ve]}ab}\int_0^1 u^2dx + C\int_{x_0-\ve}^{x_0} (u')^2dx + C\int_{x_0}^{x_0+\ve} (u')^2dx\\
&\le  \frac{2}{\ds\min_{[0,x_0-\ve]}ab}\int_0^1 u^2dx  + C\int_0^1 (u')^2dx\le \overline C_{HP} \|u\|^2_{H^1(0,1)},
\end{aligned}
\]
for a suitable positive constant $\overline C_{HP}$. Proceeding as in Corollary \ref{HPN1}.2 the claim follows immediately taking as $y_0$ the point $x_0$.
\end{proof}

\vspace{0.5cm}

Observe that, as for the Dirichlet case, the previous estimates give Hardy--Poincar\'e inequalities in all situations, namely the {\em (WWD)}, {\em (SSD)}, {\em
(WSD)} and {\em (SWD)} and Lemma \ref{L2'''} allows us to consider for the (SSD) case only the situation when $K_1= K_2=1$.

In the rest of the paper we will denote by $\mathcal C_{HP}$ one of the Hardy--Poincar\'{e} constants that appear in Proposition \ref{HP}, \ref{HPN}, Corollary \ref{HPN1} or in Lemmas \ref{L2}, \ref{L2''}, \ref{L2'} and \ref{L2'''}.

\section{Well- posedness}\label{sec3}

In order to study well-posedness of problem \eqref{linear} and in view of Lemmas \ref{L2}, \ref{L2''}, \ref{L2'} and \ref{L2'''}, we introduce the space
\[
\cK := \begin{cases} \cK_a, & \text{if Hypothesis $\ref{Ass03_new}.1$ is satisfied},\\
\cK_{a,b}, & \text{if Hypothesis $\ref{Ass03_new}.2$, $\ref{Ass03_new}.3$ or $\ref{Ass03_new}.4$ is in force},
\end{cases}
\]
where the Hardy--Poincar\'e inequality \eqref{1}, \eqref{1nondiv} or \eqref{2nondiv} holds.
\begin{Remark}\label{remultimo}
Obseve that, thanks to Lemma  \ref{L2} or \ref{L2'},  if $K_1 + K_2 <1$ the spaces $\mathcal K_a$ and $\mathcal K_{a,b}$ coincide and the two norms are equivalent in both (Dbc) or (Nbc).
\end{Remark}
\begin{Remark}\label{rem1}
If the assumptions of Lemma \ref{L2}, \ref{L2''} or \ref{L2'''} are satisfied,
then the standard norm $\|\cdot\|_{\cK}$ is equivalent to
\[
\|u\|_1^2:= \int_0^1 (u'(x))^2 dx
\]
for all $u \in \cK$. Indeed, if \eqref{1} or \eqref{2nondiv} holds, for all $u \in
\cK$, we have
\[
\int_0^1\frac{u^2}{a} dx = \int_0^1 b \frac{u^2}{ab}dx \le c \int_0^1  (u')^2
dx,
\]
for a positive constant $c$, and this is enough to conclude.
\vspace{0.3cm}
Analogously, one can prove that if \eqref{1nondiv} holds,
then the standard norm $\|\cdot\|_{\cK}$ is equivalent to
$
\|u\|_{H^1(0,1)}
$
for all $u \in \cK$.
\end{Remark}

In particular, setting $C^*$ the best constant of \eqref{1}, \eqref{1nondiv} or \eqref{2nondiv} in ${\cal K}$, one has the next result:

\begin{Corollary}\label{equi}
Assume that one among Hypothesis $\ref{Ass03_new}.2$, $\ref{Ass03_new}.3$ or $\ref{Ass03_new}.4$ is satisfied. If $(Nbc)$ hold, then for all $u \in \cK_{a,b}$  we have
\[
\ds \frac{1}{1+C^* + \max_{[0,1]}bC^*}\|u\|_{\cK_{a,b}(0,1)}^2 \le \|u'\|_{L^2(0,1)}^2 \le\max\{1,\max_{[0,1]}a \} \|u\|^2_{\cK_{a,b}(0,1)}.
\]
\end{Corollary}
\begin{proof}
Take $u \in \cK_{a,b}$ with $u'(0)=u'(1)=0$. By Lemma \ref{L2'''},
\[
\begin{aligned}
\int_0^1\frac{u^2}{a}dx&\le \max_{[0,1]}b \int_0^1\frac{u^2}{ab} dx \le \max_{[0,1]}bC^*\int_0^1 (u')^2(x) dx.
\end{aligned}
\]
Thus,
\[
\|u\|^2_{\cK_{a,b}(0,1)}= \int_0^1\frac{u^2}{a}dx +\int_0^1\frac{u^2}{ab}dx +  \int_0^1 (u')^2(x) dx \le (1+C^*+ \max_{[0,1]}bC^*) \|u'\|^2_{L^2(0,1)}.
\]
On the other hand, 
\begin{equation}\label{bo}
\begin{aligned}
\|u'\|^2_{L^2(0,1)}& \le \|u\|_{H^1(0,1)}^2
= \int_0^1 u^2 dx+ \int_0^1(u')^2dx \le \max_{[0,1]}a \int_0^1 \frac{u^2}{a} dx+ \int_0^1(u')^2dx  \\
&\le \max\{1,\max_{[0,1]}a \} \|u\|^2_{\cK_{a,b}(0,1)}.
\end{aligned}
\end{equation}
Thus, the claim follows.
\end{proof}

\vspace{0.5cm}

From now on, we make the following assumptions on $a$, $b$ and $\lambda$:
\begin{Assumptions}\label{Ass03}
\begin{enumerate}
\item 
Hypothesis $\ref{Ass03_new}$ holds;
\item
$\lambda \neq 0$ and $\lambda < \ds \frac{1}{C^*}$, if (Dbc) hold, or
\begin{equation}\label{lambda}
\lambda
< \begin{cases}
0, & \text{if Hypothesis $\ref{Ass03_new}.1$ holds}\\
\displaystyle \frac{1}{C^*}, & \text{otherwise},
\end{cases}
\end{equation}
if (Nbc) are in force.
\end{enumerate}
\end{Assumptions}
Observe that the assumption $\lambda \neq 0$ is not restrictive since the case $\lambda =0$ is considered in \cite{bfm1} and in \cite{fm1}.

Using the lemmas given in the previous section one can prove the next inequalities, which are crucial to prove well-posedness.
\begin{Proposition}\label{eq}
Assume that Hypothesis $\ref{Ass03}$ and $(Dbc)$ are satisfied. Then there exists
$\Lambda>0$ such that for all $u \in \cK$
\[
\int_0^1 (u'(x))^2 dx - \lambda \int_0^1 \frac{u^2(x)}{a(x)b(x)} dx
\ge \Lambda \|u\|^2_{\cK}.
\]
\end{Proposition}
\begin{proof} If $\lambda < 0$, the result is obvious by Remark \ref{rem1}. Now, assume that $\lambda \in
\ds \left(0,\ds \frac{1}{C^*}\right)$. Then, by
\eqref{1} and Remark \ref{rem1},
\[
\begin{aligned}
&\int_0^1 (u'(x))^2 dx - \lambda \int_0^1 \frac{u^2(x)}{a(x)b(x)} dx
\\& \ge \int_0^1 (u'(x))^2 dx - \lambda C^*\int_0^1 (u'(x))^2
dx
\\& =  (1- \lambda C^*)\int_0^1 (u'(x))^2 dx \ge
\Lambda\|u\|^2_{\cK},
\end{aligned}
\]
for a positive constant $\Lambda$.
\end{proof}

\begin{Proposition}\label{eqN}
Assume that  Hypothesis $\ref{Ass03}$ and $(Nbc)$ are satisfied. 
\begin{enumerate}
\item  If  Hypothesis $\ref{Ass03_new}.2$, $\ref{Ass03_new}.3$ or $\ref{Ass03_new}.4$ is satisfied and  $\lambda < \ds\frac{1}{C^*}$, then there exists $\Lambda >0$ such that for all $u \in \cK_{a,b}$
\[
\int_0^1 (u'(x))^2 dx - \lambda \int_0^1 \frac{u^2(x)}{a(x)b(x)} dx
\ge \Lambda \|u\|^2_{\cK_{a,b}}.
\]
\item If Hypothesis $\ref{Ass03_new}.1$  is satisfied and $\lambda  < 0$, then there exists $\Lambda >0$ such that for all $u \in \cK_{a,b}$
\[
\int_0^1 (u'(x))^2 dx - \lambda \int_0^1 \frac{u^2(x)}{a(x)b(x)} dx
\ge \Lambda \|u\|^2_{\cK_{a,b}}.
\]
\end{enumerate}
\end{Proposition}
\begin{proof}
{\it 1.:} Assume that Hypothesis $\ref{Ass03_new}.2$, $\ref{Ass03_new}.3$ or $\ref{Ass03_new}.4$ is satisfied and take $u \in \cK_{a,b}$.
If $\lambda < 0$, then, by Corollary \ref{equi}, we have
\[
\int_0^1 (u'(x))^2 dx - \lambda \int_0^1 \frac{u^2(x)}{a(x)b(x)} dx \ge \|u'\|^2_{L^2(0,1)}\ge \ds \frac{1}{1+ C^*+\max_{[0,1]}bC^*} \|u\|^2_{\cK_{a,b}}.
\]
Now, assume that $\lambda \in
\left(0, \ds \frac{1}{C^*}\right)$. By Lemma \ref{L2'''} and Corollary \ref{equi}:
\[
\begin{aligned}
&\int_0^1 (u'(x))^2 dx - \lambda \int_0^1 \frac{u^2(x)}{a(x)b(x)} dx\ge (1-\lambda C^*) \|u'\|^2_{L^2(0,1)}\\
&\ge \frac{1-\lambda C^*}{1+ C^*+\max_{[0,1]}bC^*}\|u\|^2_{\cK_{a,b}}\ge \Lambda\|u\|^2_{\cK_{a,b}},
\end{aligned}
\]
for a positive constant $\Lambda$.
\\
{\it 2.:}
Assume now that Hypothesis $\ref{Ass03_new}.1$ is satisfied and take $u \in \cK_{a,b}$. Recall that, by Remark \ref{remultimo}, $\mathcal K_a$ and $\mathcal K_{a,b}$ coincide and the two norms are equivalent. Clearly
\[
\int_0^1 \frac{u^2}{a} dx \le \max_{[0,1]}b\int_0^1\frac{u^2}{ab}dx .
\]
Being $\lambda <0$, one has
\[
- \frac{\lambda}{\max_{[0,1]}b}\int_0^1 \frac{u^2}{a} dx \le -\lambda \int_0^1\frac{u^2}{ab}dx .
\]
Hence, we get
\[
\begin{aligned}
\int_0^1 (u'(x))^2 dx - \lambda \int_0^1 \frac{u^2(x)}{a(x)b(x)} dx
&\ge  \|u'\|^2_{L^2(0,1)}- \frac{\lambda}{\max_{[0,1]}b}\int_0^1 \frac{u^2(x)}{a(x)} dx\\
&\ge\min\left\{1,- \frac{\lambda}{\max_{[0,1]}b} \right\}\|u\|^2_{\cK_a}.
\end{aligned}
\]
The thesis follows by Remark \ref{remultimo}.
\end{proof}
\begin{Remark}\label{generale}
Observe that all the previous results hold if we substitute  $(0,1)$ with  a general interval $(A,B)$ such that $x_0 \in (A,B)$.
\end{Remark}

\vspace{0.2cm}

We recall the following definition:
\begin{Definition}
Let $u_0 \in L^2_{\frac{1}{a}}(0,1)$ and $h\in
L^2_{\frac{1}{a}}(Q_T)$. A function $u$ is said to be a weak
solution of \eqref{linear} if
\[u \in C([0, T]; L^2_{\frac{1}{a}}(0,1)) \cap L^2(0, T;
\cK)\] and satisfies
\[
\begin{aligned}
&\int_0^1 \frac{ u(T,x)\varphi(T,x)}{a(x)} dx - \int_0^1
\frac{u_0(x) \varphi(0,x)}{a(x)} dx - \int_{Q_T}
\frac{\varphi_t (t,x)u(t,x)}{a(x)}dxdt =
\\&- \int_{Q_T} u_x(t,x)
\varphi_x(t,x) dxdt+\lambda\int_{Q_T} \frac{  u(t,x) \varphi(t,x)}{a(x)b(x)} dxdt\\
& + \int_{Q_T} h(t,x)\chi_\omega(x) \frac{\varphi(t,x)
}{a(x)}dx dt
\end{aligned}
\]
for all $\varphi \in H^1(0, T; L^2_{\frac{1}{a}}(0,1)) \cap L^2(0,
T; \cK)$.
\end{Definition}
Finally, we introduce the Hilbert spaces

\[
\cW:= \begin{cases}
H^2_{\frac{1}{a},\frac{1}{b}}(0,1) :=\Big\{ u \in H^1_{\frac{1}{a}}(0,1) \big| \;u' \in H^1(0,1) \text{ and }\; A_1u \in
L^2_{\frac{1}{a}}(0,1)\Big\}, & \text{if (Dbc) hold},\\
\cH^2_{\frac{1}{a},\frac{1}{b}}(0,1) :=\Big\{ u \in \cH^1_{\frac{1}{a}}(0,1) \big| \;u' \in H^1(0,1) \text{ and }\; A_2u \in
L^2_{\frac{1}{a}}(0,1)\Big\}, & \text{if (Nbc) are in force},
\end{cases}
\]
where $A_iu:=au''  + \displaystyle
\frac{\lambda}{b}u$, $i=1,2$, with
\[D(A_1)
=H^2_{\frac{1}{a}, \frac{1}{b}}, \quad \text {if (Dbc) hold},
\]
and
\[D(A_2) =\left\{u \in \cH^2_{\frac{1}{a},\frac{1}{b}}:
u'(0)=u'(1)=0\right\}, \quad \text {if (Nbc) are in force}.
\]

\begin{Remark}
Observe that if $u \in D(A_i)$, $i=1,2$, then $\ds\frac{u}{b}$ and $\ds\frac{u}{\sqrt{b}}\in L^2_{\frac{1}{a}}(0,1)$, so that $u \in \cK_{a,b}$ and \eqref{1}, \eqref{1nondiv} or \eqref{2nondiv} holds if Hypothesis $\ref{Ass03_new}$ is satisfied.
\end{Remark}

As in \cite[Lemma 2.2]{fm1}, one can prove the following
formula of integration by parts which is a crucial tool for the rest of the paper:
\begin{Lemma}\label{green} Assume that one among the Hypothesis $\ref{Ass0}$, $\ref{Ass01}$, $\ref{Ass0_1}$ or $\ref{Ass01_1}$ is satisfied. Then, for all $(u,v)\in D(A_i)\times \cK$, $i=1,2$,
 one has
\begin{equation}\label{greenformula}
\int_0^1u'' v dx= - \int_0^1 u'v' dx.
\end{equation}
\end{Lemma}

The following existence result holds:
\begin{Theorem}\label{theorem1}
Assume Hypothesis $\ref{Ass03}$.
For
all $h \in  L^2_{\frac{1}{a}}(Q_T)$ and $u_0 \in L^2_{\frac{1}{a}}(0,1)$, there exists a unique
weak solution  $u$ of \eqref{linear}. In particular, the operators $A_i: D(A_i) \to
L^2_{\frac{1}{a}}(0,1)$, $i=1,2$, are non positive and self-adjoint in $L^2_{\frac{1}{a}}(0,1)$ and generate two analytic contraction
semigroups of angle $\pi/2$. Moreover, if $u_0 \in D(A_i)$ and $h\in W^{1,1}(0,T;L^2_{\frac{1}{a}}(0,1))$, then
\[
u \in C^1(0,T; L^2_{\frac{1}{a}}(0,1)) \cap C([0,T];D(A_i)).
\]
\end{Theorem}

Observe that in the non degenerate case we know that the heat operator with an inverse-square singular potential gives rise to well posed Cauchy problems if and only if the parameter $\lambda$ that appears in \eqref{linear} is not larger than the best Hardy inequality (see, for example, \cite{vz2}). For this reason, it is not strange that also in this case we require an analogous condition for \eqref{linear} (for other comments see \cite{fm2}).

We recall that the case $\lambda =0$ is considered in \cite{bfm1} and in \cite{fm1} when Neumann or Dirichlet boundary conditions hold, respectively.

\begin{proof}[Proof of Theorem  \ref{theorem1}]$\quad$

{\em\underline{If Dirichlet boundary conditions hold:} }
 Observe that $D(A_1)$ is dense in $L^2_{\frac{1}{a}}(0,1)$.  We will proceed as in \cite{fm2} proving that $A_1$ is
nonpositive, self-adjoint and  hence $m-$ dissipative.
\\
\textbf{$\boldsymbol {A_1}$ is nonpositive.} By Proposition \ref{eq},
for all $u \in D(A_1)$ we have
\[
\begin{aligned}
- \langle A_1u, u \rangle_{L^2_{\frac{1}{a}}(0,1)} = -\int_0^1\left( au'' +
\frac{\lambda}{b}u\right)\frac{u}{a} dx = \int_0^1  (u')^2 dx - \lambda
\int_0^1 \frac{u^2}{ab}dx
 \ge \Lambda
\|u\|_{\cK}^2 \ge 0,
\end{aligned}
\]
which proves the result.
\\
\textbf{$\boldsymbol {A_1}$ is self-adjoint.} Let $T:L^2_{\frac{1}{a}}(0,1)\to L^2_{\frac{1}{a}}(0,1)$ be the mapping defined in the following usual way: to each $h\in L^2_{\frac{1}{a}}(0,1)$ we associate the weak solution $u=T(h)\in {\cal K}$ of 
\[
\int_0^1 \left(u'v' - \lambda \frac{u v}{ab}
\right) dx= \int_0^1 \frac{h v}{a} dx
\]
for every $v\in {\cal K}$. Note that $T$ is well defined by the Lax--Milgram Lemma via Proposition \ref{eq}, which also implies that $T$ is continuous. Now, it is easy to see that $T$ is injective and symmetric. Thus it is self--adjoint. As a consequence, $A_1=T^{-1}:D(A_1)\to L^2_{\frac{1}{a}}(0,1)$ is self--adjoint (for example, see \cite[Proposition A.8.2]{Taylor}).
\\
\textbf{$\boldsymbol{A_1}$ is $\boldsymbol{m}$--dissipative}. Being $A_1$ nonpositive and self--adjoint, the $m-$ dissipativity of the operator is a straightforward consequence of \cite[Corollary 2.4.8]{ch}.
\\
Hence $(A_1, D(A_1))$ generates a cosine family
and an analytic contractive semigroup of angle $\ds\frac{\pi}{2}$ on
$L^2_{\frac{1}{a}}(0,1)$ (see, for example, \cite[Example 3.14.16 and
3.7.5]{abhn}).
\\
The additional regularity is a consequence of \cite[Lemma 4.1.5 and Proposition 4.1.6]{ch}.

{\em  \underline {If Neumann boundary conditions hold: } }
The proof in this case is similar to the previous one, but the nonpositivity of the operator $A_2$ and the wellposedness of $T$ follow by Proposition \ref{eqN}.
\end{proof}

\section{Carleman estimates for the adjoint problem}\label{sec4}
In this section we prove one of the main result of this paper, i.e.
new Carleman estimates for solutions of
 the following problem:
\[
\begin{cases}
v_t +a(x)v_{xx}  + \displaystyle \frac{\lambda}{b(x)}v=h(t,x)=h, & (t,x) \in Q_T,\\
Bv(0)=Bv(1)=0, &t \in (0,T),\\
v(T,x)= v_T(x)\in L^2_{\frac{1}{a}}(0,1),
\end{cases}
\]
which is the adjoint problem of \eqref{linear}. Here $T>0$ is given. As it is well known, to prove Carleman estimates the final datum is irrelevant, only the equation and the boundary conditions are important. For this reason we can consider only the problem

\begin{equation}\label{P-adjoint}
\begin{cases}
v_t +a(x)v_{xx}  + \displaystyle \frac{\lambda}{b(x)}v=h(t,x)=h, & (t,x) \in Q_T,\\
Bv(0)=Bv(1)=0, &t \in (0,T).
\end{cases}
\end{equation}

\vspace{0.5cm}

First of all, we will consider the case when $a$ and $b$ are strictly positive, since it will be crucial in the next section to prove observability inequalities.
On $a$ and $b$ we make the following assumptions:
\begin{Assumptions}\label{iponondegenere} There exist two strictly positive constants $a_0, b_0$ such that $a\geq
a_0$ and $b \ge b_0$ in $[0,1]$. Moreover, $b \in C\big([0,1]\big)$
 and $a$ satisfies
\begin{itemize}
\item[$(a_1)$] $a\in W^{1,1}(0,1)$,  there exist two functions $\fg \in L^1(0,1)$,
$\fh \in W^{1,\infty}(0,1)$ and two strictly positive constants
$\fg_0$, $\fh_0$ such that $\fg(x) \ge \fg_0$ and
\[\frac{a'(x)}{2\sqrt{a(x)}}\left(\int_x^{1}\fg(t) dt + \fh_0 \right)+ \sqrt{a(x)}\fg(x) =\fh(x)\quad \text{for a.e.} \; x \in [0,1];\]
or
\item[$(a_2)$] $a\in W^{1,\infty}(0,1)$.
\end{itemize}
\end{Assumptions}

Now, define
\begin{equation}\label{theta}
\Theta(t):=\frac{1}{[t(T-t)]^4}\quad (\rightarrow + \infty \,
\text{ as } t \rightarrow 0^+, T^-)
\end{equation}
and
\begin{equation}\label{c_1nd}
\rho_{0,1}(x):=
\displaystyle\begin{cases} \displaystyle - r\left[\int_0^x
\frac{1}{\sqrt{a(t)}} \int_t^1
\fg(s) dsdt + \int_0^x \frac{\fh_0}{\sqrt{a(t)}}dt\right] -\mathfrak{c}, &\text{ if  ($a_1$) holds,}\\
\displaystyle  e^{r\zeta_1(x)}-\mathfrak{c}, &\text{ if  ($a_2$)
holds},\end{cases}
\end{equation}
where
 \[
\zeta_1(x)=\mathfrak{d}\int_x^{1}\frac{1}{a(t)}dt,
\]
$\fd=\|a'\|_{L^\infty(0,1)}$ and $\mathfrak{c}>0$ is
chosen in the second case in such a way that $\displaystyle
\max_{[0,1]} \rho_{0,1}<0$.

\begin{Proposition}[\bf Nondegenerate Carleman estimate]\label{classical
Carleman} Assume Hypothesis $\ref{iponondegenere}$.
Let $z$ solves the non degenerate system
\begin{equation}
    \label{NDeq-z*}
    \begin{cases}
      z_t + az_{xx} + \lambda \displaystyle \frac{z}{b}= h\in L^2\big(Q_T\big),
         \\
    Bz(0)= Bz(1)=0, \; t \in (0,T).  
    \end{cases}
    \end{equation}
Then, if Dirichlet boundary conditions hold, there exist three positive constants $C$, $r$ and $s_0$
such that for any $s>s_0$
\begin{equation}\label{carcorretta}
\begin{aligned}
\int_{Q_T} \left(s\Theta (z_x)^2 + s^3 \Theta^3
 z^2\right)e^{2s\Phi}dxdt\le C\left(\int_{Q_T} h^{2}e^{2s\Phi}dxdt -
(\mbox{B.T.})
\right),
\end{aligned}
\end{equation}
where
\begin{equation}\label{BT}
(\mbox{B.T.}) = \begin{cases} \ds sr\int_0^T\Theta(t)\left[\sqrt{a}\left(\int_x^{1} \fg(\tau) d\tau +
\fh_0 \right)(z_x)^2 e^{2s\Phi}\right]_{x=0}^{x=1}dt,  & \text{if $(a_1)$ holds,}\\
\ds sr\int_0^T\left[ae^{2s\Phi}\Theta e^{r\zeta_1}(z_x)^2
\right]_{x=0}^{x=1}dt,  & \text{if $(a_2)$ holds.}
\end{cases}
\end{equation}
If Neumann boundary conditions hold and $(\sigma, \gamma) \subset (0,1)$, then there exist three positive constants $C$ (depending on $\sigma$ and $\gamma$), $r$ and $s_0$
such that for any $s>s_0$
\begin{equation}\label{carcorrettaN}
\begin{aligned}
\int_{Q_T} \left(s\Theta (z_x)^2 + s^3 \Theta^3
 z^2\right)e^{2s\Phi}dxdt\le C\left(\int_{Q_T} h^{2}e^{2s\Phi}dxdt + \int_0^T\int_\sigma^\gamma z^2e^{2s\Phi}dxdt
\right).
\end{aligned}
\end{equation}

Here $\Phi(t,x):
=\Theta(t)\rho_{0,1}(x)$.
\end{Proposition}
(Observe that $\Phi <0$ and $\Phi(t,x) \rightarrow -\infty$, as $t \rightarrow 0^+, T^-$.)

\begin{proof}
$\quad$

{\it\underline{If Dirichlet boundary conditions hold:}}
     Rewrite the equation 
  satisfied by $z$  as $ z_t + az_{xx} = \bar{h}, $ where $\bar{h}
    := h - \lambda \displaystyle \frac{z}{b} $. Then, applying \cite[Theorem 3.1]{fm1}, there exist
three positive constants $C$, $r$ and $s_0 >0$, such that
    \begin{equation}\label{fati1_c}
    \begin{aligned}
 \int_{Q_T} \left(s\Theta (z_x)^2 + s^3 \Theta^3
 z^2\right)e^{2s\Phi}dxdt\le C\left(\int_{Q_T} {\bar h}^{2}e^{2s\Phi}dxdt -
\text{(B.T.)}\right),
    \end{aligned}
    \end{equation}
    for all $s \ge s_0$. Here the boundary terms (B.T.) are as in \eqref{BT}.
    Using the definition of $\bar{h}$,
    the term $\int_{Q_T}\bar{h}^2e^{2s\Phi}
    dxdt$ can be estimated in the following way:
    \begin{equation}\label{4_c}
    \begin{aligned}
    \int_{Q_T}  \bar{h}^2e^{2s\Phi}dxdt \le
    2\int_{Q_T} h^2 e^{2s\Phi}dxdt
   +2\lambda^2\int_{Q_T}
    \frac{z^2}{b^2} e^{2s\Phi}dxdt.
    \end{aligned}
    \end{equation}
Now, we proceed as in \cite[Proposition 4.3]{fm2}:
applying the classical Poincar\'{e} inequality to $w(t,x) :=
e^{s\Phi} z(t,x)$ and observing that
$0<\inf \Theta\leq \Theta\leq c \Theta^2$, one has
    \[
    \begin{aligned}
 2\lambda^2 \int_{Q_T}\frac{z^2}{b^2}  e^{2s\Phi}dxdt &= 2\lambda^2\int_{Q_T}
\frac{w^2}{b^2} dxdt\le 2\frac{\lambda^2}{b_0^2}C \int_{Q_T}  (w_x)^2 dxdt
\\&\leq C \int_{Q_T}  (s^2\Theta^2  z^2+(z_x)^2 )e^{2s\Phi} dxdt \\
&\le \int_{Q_T} \frac{s}{2} \Theta (z_x)^2e^{2s\Phi} dxdt+\int_{Q_T} \frac{s^3}{2}\Theta^3 z^2e^{2s\Phi} dxdt,
 \end{aligned}\]
for $s$ large enough.
    Substituting this inequality in (\ref{4_c}), we have
    \[
    \begin{aligned}
  \int_{Q_T} \bar{h}^2e^{2s\Phi}dxdt&\le
2\int_{Q_T} h^2e^{2s\Phi}dxdt
   +
 \int_{Q_T}\frac{s}{2} \Theta (z_x)^2e^{2s\Phi}
dxdt +\int_{Q_T}\frac{s^3}{2}\Theta^3 z^2e^{2s\Phi} dxdt.
    \end{aligned}
    \]
    Using the last inequality
    in (\ref{fati1_c}), \eqref{carcorretta} follows immediately.

{\it\underline{If Neumann boundary conditions hold:}} We will use a reflection procedure. Consider a
smooth function $\xi: [-1,2] \to \R$ such that $\xi \equiv 1$ in $[-\sigma, 1+\sigma]$ and $ \xi \equiv 0$ in $ [-1, -\gamma] \cup [1+\gamma, 2]$.
Now, define
\begin{equation}\label{W}
W(t,x):= \begin{cases}
z(t, 2-x), &x \in [1,2],\\
z(t,x), & x \in [0,1],\\
z(t,-x), & x \in [-1,0],
\end{cases}
\end{equation}
where $z$ solves \eqref{NDeq-z*}.
Thus $W$ satisfies
\begin{equation}\label{PW}
\begin{cases}
W_t + \tilde a W_{xx}+ \lambda \ds \frac{W}{\tilde b} =\tilde h, & (t,x) \in (0,T) \times (-1,2),\\
W_x(t,-1)=W_x(t,2)=0, &  t \in (0,T),\\
\end{cases}
\end{equation}
being
\begin{equation}\label{tildea}
\tilde a(x):= \begin{cases}
a(2-x),& x \in [1,2],\\
a(x), & x \in [0,1],\\
a(-x), & x \in [-1,0],
\end{cases} \quad
\tilde b(x):= \begin{cases}
b(2-x),& x \in [1,2],\\
b(x), & x \in [0,1],\\
b(-x), & x \in [-1,0]
\end{cases}
\end{equation}
and
\begin{equation}\label{tildeh}
\tilde h(t,x) := \begin{cases}
h(t, 2-x), &x \in [1,2],\\
h(t,x), & x \in [0,1],\\
h(t,-x), & x \in [-1,0].
\end{cases}
\end{equation}
Observe that $\tilde a, \tilde b$ belong to $W^{1,1}(-1,2)$ or  to $W^{1,\infty}(-1,2)$, if $a, b$ belong to  $W^{1,1}(0,1)$ or  to $W^{1,\infty}(0,1)$, respectively.
Now, set $Z:= \xi W$; then $Z$ solves
\[
\begin{cases}
Z_t + \tilde a Z_{xx} =H, & (t,x) \in (0,T) \times (-1,2),\\
Z(t,-1)=Z(t,2)=0, &  t \in (0,T),\\
\end{cases}
\]
 with  $H:=\xi \tilde h+ \tilde a (\xi _{xx} W + 2\xi _x  W_x)-\ds\lambda \xi \frac{W}{\tilde b} $. Observe that $Z_x(t, -1)=Z_x(t, 2)=0$ and, by the assumption on $a$, $H\in L^2((0,T); L^2_{\frac{1}{\tilde a}} (-1,2)).$
Now, define
$\tilde \Phi(t,x) := \Theta(t) \rho_{-1,2} (x),$ with
\[
\rho_{-1,2}(x):=
\displaystyle - r\left[\int_{-1}^x
\frac{1}{\sqrt{a(t)}} \int_t^2
\fg(s) dsdt + \int_{-1}^x \frac{\fh_0}{\sqrt{a(t)}}dt\right] -\mathfrak{c}, 
\]
if the analogous of  ($a_1$) holds for  $\tilde a \text{ in } [-1, 2],$ and
\[
\rho_{-1,2}(x):=
\displaystyle  e^{r\zeta_2(x)}-\mathfrak{c}, 
\]
if the analogous of  ($a_2$) is in force for
$\tilde a\text{ in } [-1, 2]$. Here
 \[
\zeta_2(x)=\mathfrak{d}\int_x^{2}\frac{1}{\tilde a(t)}dt,
\]
$\fd=\|\tilde a'\|_{L^\infty(-1,2)}$ and $\mathfrak{c}>0$ is
chosen in the second case in such a way that $\displaystyle
\max_{[-1,2]} \rho_{-1,2}<0$.

Thus, we can apply the
analogue of \cite[Theorem 3.2]{fm1} on $(-1,2)$ in place of
$(A,B)$ and with weight $\tilde \Phi$, obtaining that there exist
two positive constants $C$ (depending on $\sigma$ and $\gamma$) and $s_0$ ($s_0$ sufficiently large), such that $Z$
satisfies, for all $s \ge s_0$,
\[
\begin{aligned}
&\int_0^T\int_{-1}^2 \left(s\Theta (Z_x)^2 + s^3 \Theta^3
Z^2\right)e^{2s\tilde \Phi}dxdt\\
&\le C\left(\int_0^T\int_{-1}^2 H^{2}e^{2s\tilde \Phi}dxdt -
sr\int_0^T\Theta(t)\left[\sqrt{a}\left(\int_x^2 \fg(\tau) d\tau +
\fh_0 \right)(Z_x)^2 e^{2s\tilde \Phi}\right]_{x=-1}^{x=2}dt\right)\\
&= C\int_0^T\int_{-1}^2 H^{2}e^{2s\tilde \Phi}dxdt= C\int_0^T\int_{-1}^{2}\left(\xi \tilde h+ \tilde a (\xi _{xx} W + 2\xi _x  W_x)-\ds\lambda \xi \frac{W}{\tilde b}\right)^2e^{2s\tilde \Phi}dxdt\\
&\le C\int_0^T\int_{-1}^{2} \left( \tilde h^2+ \lambda ^2 \left(\frac{W}{\tilde b}\right)^2\right) dxdt \\
&+ C \left( \int_0^T\int_{-\gamma}^{-\sigma}(W^2 +
W_x^2)dxdt+ \int_0^T\int_{1+\sigma}^{1+\gamma}(W^2 +
W_x^2)dxdt\right).
\end{aligned}
\]

Hence, by definitions of $Z$, $W$ and $\tilde b$,
we have
\[
\begin{aligned}
&\int_{Q_T}\left(s\Theta (z_x)^2 + s^3 \Theta^3
z^2\right)e^{2s\Phi}dxdt
\le
\int_0^T\int_{-1}^2 \left(s\Theta (Z_x)^2 + s^3 \Theta^3
Z^2\right)e^{2s\tilde \Phi}dxdt\\
&\le C \int_0^T\int_{-1}^{2} \left( \tilde h^2+ \lambda ^2 \left(\frac{W}{\tilde b}\right)^2\right) dxdt\\
&+ C \left( \int_0^T\int_{-\gamma}^{-\sigma}(W^2 +
W_x^2)dxdt+ \int_0^T\int_{1+\sigma}^{1+\gamma}(W^2 +
W_x^2)dxdt\right)\\
&\le C\left(\int_{Q_T}\left( h^2+ \lambda ^2 \left(\frac{z}{b}\right)^2\right)e^{2s\Phi}dxdt+ \int_0^T\int_\sigma^\gamma (z^2+z_x^2) e^{2s\Phi}dxdt\right)\\
& \le C\left( \int_{Q_T} h^2 e^{2s\Phi}dxdt+  \frac{\lambda ^2}{b_0}\int_{Q_T} z^2e^{2s\Phi}dxdt+ \int_0^T\int_\sigma^\gamma (z^2+z_x^2) e^{2s\Phi}dxdt\right)\\
&\le C\left( \int_{Q_T} \frac{s^3}{2} \Theta^3 z^2e^{2s\Phi} dxdt + \int_{Q_T} \frac{s}{2} \Theta( z_x)^2e^{2s\Phi} dxdt\right)\\
 &+ C\left( \int_{Q_T} h^2 e^{2s\Phi}dxdt+ \int_0^T\int_\sigma^\gamma z^2e^{2s\Phi}dxdt\right)
\end{aligned}
\]
for $s$ large enough and for a positive constant $C$ depending on $\sigma$ and $\gamma$. Thus the claim follows immediately.
\end{proof}

\begin{Remark}\label{rembo} We underline that Proposition \ref{classical Carleman} still holds if we substitute the spatial domain $[0,1]$ with a general interval $[\mathcal A, \mathcal B]$ where the functions $a$ and $b$ satisfy Hypothesis \ref{iponondegenere}.
\end{Remark}
\vspace{0.5cm}

In the following we will assume that the functions $a$ and $b$ are zero at $x_0$. In particular, we make the following assumptions:
\begin{Assumptions}\label{Ass02}
\begin{enumerate}
\item Hypothesis $\ref{Ass03}$ is satisfied;
\item $\ds
\frac{(x-x_0)a'(x)}{a(x)} \in
W^{1,\infty}(0,1);$
\item if $K_1 \ge\ds \frac{1}{2}$, there exists $\theta \in (0, K_1]$ such that the function $x \mapsto \ds\frac{a}{|x-x_0|^\theta}$ is nonincreasing on the left  and nondecreasing on the right of $x= x_0$;
\item  if $\lambda <0$, then
$
(x-x_0)b'(x) \ge 0  \text{ in } [0,1]$.
\end{enumerate}
\end{Assumptions}

To prove Carleman estimates, let us introduce the function $\varphi:= \Theta \psi$, where $\Theta$ is as in \eqref{theta} and
\begin{equation}\label{c_1}
\psi(x) := d_1\left(\int_{x_0}^x \frac{y-x_0}{a(y)}e^{R(y-x_0)^2}dy-
d_2\right).
\end{equation}
Here $R$ and $d_1$ are general strictly positive constants, while
\[d_2> \displaystyle
\max\left\{\frac{(1-x_0)^2e^{R(1-x_0)^2}}{(2-K)a(1)},
\frac{x_0^2e^{Rx_0^2}}{(2-K)a(0)}\right\}.\]
The choice of $d_2$ implies immediately, by Lemma \ref{Lemma 2.1}, that
\[
-d_1d_2\le\psi(x)<0 \quad \mbox{ for every }x\in[0,1].
\]
\vspace{0.5cm}

Now, define the space
\begin{equation}\label{v1}
\mathcal V:= H^1\big(0, T;\cK) \cap
L^2\big(0, T; \cW \big).
\end{equation}
The main result of this section is the following:
\begin{Theorem}\label{Cor1}
Assume Hypothesis $\ref{Ass02}$.  There exist two positive constants $C$ and $s_0$ (depending on $\lambda$) such that every
solution $v$ of \eqref{P-adjoint} in $\mathcal V$
satisfies, for all $s \ge s_0$,
\begin{equation}\label{car}
\begin{aligned}
&\int_{Q_T}\!\!\! \left(s\Theta (v_x)^2 + s^3 \Theta^3
\left(\frac{x-x_0}{a} \right)^2v^2\right)e^{2s\varphi}dxdt
\\
&\le C\left(\int_{Q_T} h^{2}\frac{e^{2s\varphi}}{a}dxdt+sd_1\int_0^T\Theta\Big[(x-x_0)e^{R(x-x_0)^2}(w_x)^2\Big]_{x=0}^{x=1}dt\right)
\end{aligned}
\end{equation}
if (Dbc) hold and
\begin{equation}\label{carN}
\begin{aligned}
&\int_{Q_T} \!\!\!\left(s\Theta (v_x)^2 + s^3 \Theta^3
\left(\frac{x-x_0}{a} \right)^2v^2\right)e^{2s\varphi}dxdt
\le C\left(\int_{Q_T} h^{2}\frac{e^{2s\varphi}}{a}dxdt +
 \int_{Q_T} v^{2}e^{2s\varphi}dxdt\right),
\end{aligned}
\end{equation}
if (Nbc) are in force.

More precisely, if $\omega$ is a strict subset of $(0,1)$ such that $x_0 \in \omega$, then \eqref{carN} becomes
\begin{equation}\label{carN_1}
\begin{aligned}
&\int_{Q_T} \!\!\!\left(s\Theta (v_x)^2 + s^3 \Theta^3
\left(\frac{x-x_0}{a} \right)^2v^2\right)e^{2s\varphi}dxdt\le C\left(\int_{Q_T} h^{2}\frac{e^{2s\varphi}}{a}dxdt +\int_0^T \int_\omega v^2e^{2s\varphi}dxdt\right).
\end{aligned}
\end{equation}
Here $d_{1}$ is the constant introduced in
\eqref{c_1}.
\end{Theorem}

\subsection{Proof of Theorem  \ref{Cor1}  if (Dbc) hold}
For the proof of Theorem \ref{Cor1} we proceed as in \cite{fm2}. First, for $s> 0$, define the function
\[
w(t,x) := e^{s \varphi (t,x)}v(t,x),
\]
where $v$ is any solution of \eqref{P-adjoint} in $\mathcal{V}$;
observe that, since $v\in\mathcal{V}$ and $\vp<0$, then
$w\in\mathcal{V}$ and satisfies
\begin{equation}\label{1'}
\begin{cases}
(e^{-s\varphi}w)_t + a(e^{-s\varphi}w)_{xx}+ \lambda \displaystyle\frac{e^{-s\varphi}w} {b} =h, & (t,x) \in Q_T,\\
w(t,0)=w(t,1)=0, &  t \in (0,T),\\ w(T,x)= w(0, x)= 0, & x \in
(0,1).
\end{cases}
\end{equation}
As usual, we re--write the previous problem as
follows: setting
\[
Lv:= v_t + (av_x)_x + \lambda \frac{v}{b} \quad \text{and} \quad
%\]
%and
%\[
L_sw= e^{s\varphi}L(e^{-s\varphi}w),
\]
then \eqref{1'} becomes
\[
\begin{cases}
L_sw= e^{s\varphi}h, & (t,x) \in Q_T,\\
w(t,0)=w(t,1)=0, & t \in (0,T),\\
w(T,x)= w(0, x)= 0, & x \in (0,1).
\end{cases}
\]
Computing $L_sw$, one has
\[
\begin{aligned}
L_sw
=L^+_sw + L^-_sw,
\end{aligned}
\]
where
\[
L^+_sw := aw_{xx} + \lambda\frac{w}{b}
 - s \varphi_t w + s^2a \varphi_x^2 w,
\]
and
\[
L^-_sw := w_t -2sa\varphi_x w_x -
 sa\varphi_{xx}w.
\]
Of course,
\begin{equation}\label{stimetta}
\begin{aligned}
2\langle L^+_sw, L^-_sw\rangle_{L^2_{\frac{1}{a}}(Q_T)} &\le 2\langle L^+_sw, L^-_sw\rangle_{L^2_{\frac{1}{a}}(Q_T)}+
\|L^+_sw \|_{L^2_{\frac{1}{a}}(Q_T)}^2 + \|L^-_sw\|_{L^2_{\frac{1}{a}}(Q_T)}^2\\
& =\| L_sw\|_{L^2_{\frac{1}{a}}(Q_T)}^2= \|he^{s\varphi}\|_{L^2_{\frac{1}{a}}(Q_T)}^2
\end{aligned}
\end{equation}
 Proceeding as in \cite{fm1} and in \cite{fm2}, we will separate the scalar product $\langle
L^+_sw, L^-_sw\rangle_{L^2_{\frac{1}{a}}(Q_T)}$ in distributed terms and boundary terms:
\vspace{0.2cm}

\begin{Lemma}\label{lemma1}
The following identity holds:
\begin{equation}\label{D&BT}
\begin{aligned}
& \left.
\begin{aligned}
&\langle L^+_sw,L^-_sw\rangle_{L^2_{\frac{1}{a}}(Q_T)} \;\\&=\; \frac{s}{2} \int_{Q_T}
\frac{\varphi_{tt}}{a} w^2dxdt 
- 2s^2 \int_{Q_T} \varphi_x \varphi_{tx}w^2dxdt \\& +s
\int_{Q_T}(2 a\varphi_{xx} + a'\varphi_x)(w_x)^2 dxdt
\\&+ s^3 \int_{Q_T}(2a \varphi_{xx} + a'\varphi_x)
(\varphi_x)^2 w^2dxdt
\\&+s\int_{Q_T}(a\varphi_{xx})_{x}ww_x dxdt-s\lambda
\int_{Q_T} \frac{\varphi_xb'}{b^2} w^2 dxdt
\end{aligned}\right\}\;\text{\{D.T.\}} \\
& \left.
\begin{aligned}
& + \int_0^T[w_xw_t]_{x=0}^{x=1} dt- \frac{s}{2}
\int_0^1\left[w^2\frac{\varphi_t}{a}\right]_{t=0}^{t=T}dx+ \frac{s^2}{2}\int_0^1
[(\varphi_x)^2 w^2]_{t=0}^{t=T}dx\\& + \int_0^T[-s\varphi_x
a(w_x)^2 +s^2\varphi_t \varphi_x w^2 - s^3 a(\varphi_x)^3w^2-
s\lambda \frac{\varphi_x }{b} w^2]_{x=0}^{x=1}dt
\\&+ \int_0^T[-sa\varphi_{xx}w
w_x]_{x=0}^{x=1}dt-\frac{1}{2} \int_0^1
\Big[(w_x)^2-\lambda\frac{1}{2ab}w^2 \Big]_{t=0}^{t=T}dx.
\end{aligned}\right\}\; \text{\{B.T.\}}
\end{aligned}
\end{equation}
\end{Lemma}
\begin{proof} Computing $\langle L^+_sw,L^-_sw\rangle_{L^2_{\frac{1}{a}}(Q_T)}$,
one has that
\[
\langle L^+_sw,L^-_sw\rangle_{L^2_{\frac{1}{a}}(Q_T)} = I_1+ I_2 + I_3+I_4,
\]
where
\[
\begin{aligned}
I_1& := \int_{Q_T}\big(aw_{xx}
 - s \varphi_t w + s^2a (\varphi_x)^2 w\big) \frac{w_t}{a}  dxdt,\\
I_2& := \int_{Q_T}\big(aw_{xx}
 - s \varphi_t w + s^2a (\varphi_x)^2 w\big) (-2s\varphi_xw_x)  dxdt,\\
I_3& := \int_{Q_T}\big(aw_{xx}
 - s \varphi_t w + s^2a (\varphi_x)^2 w\big)(-s\varphi_{xx}w) dxdt,
\end{aligned}
\]
and
\[
I_4:= \lambda \int_{Q_T} \frac{w}{ab}\big( w_t -2sa\varphi_x
w_x -
 sa\varphi_{xx}w\big)dxdt.
\]
It is sufficient to compute $I_4$, since $I_1+ I_2+ I_3$ follows as in \cite{fm1}. Hence
\begin{equation}\label{T_4}
\begin{aligned}
I_4&= \lambda \left( \int_0^1 \frac{1}{2ab}[w^2]_{t=0}^{t=T} dx -  s
\int_{Q_T} \frac{1}{b} \varphi_x (w^2)_x dxdt -s
\int_{Q_T} \frac{\varphi_{xx}}{b} w^2 dxdt\right)\\
&= \lambda \left( \int_0^1 \frac{1}{2ab}[w^2]_{t=0}^{t=T} dx -  s
\int_0^T\left[\frac{1}{b} \varphi_x w^2\right]_{x=0}^{x=1}dt\right.\\
& \left. + s \int_{Q_T} \left(\frac{\varphi_x}{b}\right)_x
w^2 dxdt-s
\int_{Q_T} \frac{\varphi_{xx}}{b} w^2 dxdt\right)\\
&=\lambda \left( \int_0^1 \frac{1}{2ab}[w^2]_{t=0}^{t=T} dx -  s
\int_0^T\left[\frac{\varphi_x}{b} w^2\right]_{x=0}^{x=1}dt -s
\int_{Q_T} \frac{\varphi_xb'}{b^2} w^2 dxdt\right).
\end{aligned}
\end{equation}
\end{proof}
For the boundary terms in \eqref{D&BT}, we have:
\begin{Lemma}\label{lemma4}
The boundary terms in Lemma \ref{lemma1} reduce to
\[-sd_1\int_0^T\Theta(t)\Big[(x-x_0)e^{R(x-x_0)^2}(w_x)^2\Big]_{x=0}^{x=1}dt.
\]
\end{Lemma}
\begin{proof}
As in \cite[Lemma 4.4]{fm1}, using the definition of $\varphi$ and the boundary
conditions on $w$, one has that the boundary terms in
\eqref{D&BT}, without considering the terms $\ds\lambda\int_0^1\left[\frac{1}{2ab}w^2\right]_{t=0}^{t=T} dx$ and  $\displaystyle s
\lambda \int_0^T
\left[\frac{\varphi_x}{b}w^2\right]_{x=0}^{x=1}dt$, reduce to
\[-sd_1\int_0^T\Theta(t)\Big[(x-x_0)e^{R(x-x_0)^2}(w_x)^2\Big]_{x=0}^{x=1}dt.
\]
Moreover, since $w \in \mathcal{V}$, $ w\in C\big([0, T];{\cal K}
\big)$; thus $w(0, x)$, $w(T,x)$ are well defined, and using again
the boundary conditions of $w$, we get that
\[
\int_0^1\left[\frac{1}{2ab}w^2\right]_{t=0}^{t=T} dx=0.
\]
Now, consider the last boundary term $\displaystyle s \lambda
\int_0^T \left[\frac{\varphi_x}{b}w^2\right]_{x=0}^{x=1}dt$. Using
the definition of $\varphi$, this term becomes $\displaystyle s
\lambda \int_0^T
\left[\Theta\frac{\psi'}{b}w^2\right]_{x=0}^{x=1}dt. $  By
definition of $\psi$, the
function $\displaystyle \Theta \frac{\psi'}{b}w^2$ is bounded in $(0,T)$.
Thus, by the boundary conditions on $w$, one has
\[
s \lambda \int_0^T
\left[\Theta\frac{\psi'}{b}w^2\right]_{x=0}^{x=1}dt=0.
\]
\end{proof}

Now, the crucial step is to prove the following estimate:
\begin{Lemma}\label{lemma2}
Assume Hypothesis $\ref{Ass02}$. Then there exist two positive constants $C$ (depending on $\lambda$) and
$s_0$ such that for
all $s \ge s_{0}$ the distributed terms of \eqref{D&BT} satisfy the
estimate
\[
\begin{aligned}
&\frac{s}{2} \int_{Q_T}
\frac{\varphi_{tt}}{a} w^2dxdt 
- 2s^2 \int_{Q_T} \varphi_x \varphi_{tx}w^2dxdt  +s
\int_{Q_T}(2 a\varphi_{xx} + a'\varphi_x)(w_x)^2 dxdt
\\&+ s^3 \int_{Q_T}(2a \varphi_{xx} + a'\varphi_x)
(\varphi_x)^2 w^2dxdt
+s\int_{Q_T}(a\varphi_{xx})_{x}ww_x dxdt-s\lambda
\int_{Q_T} \frac{\varphi_xb'}{b^2} w^2 dxdt\\&\ge \frac{C}{2}s\int_{Q_T} \Theta (w_x)^2 dxdt +
\frac{C^3}{2}s^3 \int_{Q_T}\Theta^3 \left(\frac{x-x_0}{a} \right)^2w^2
dxdt.
\end{aligned}
\]
\end{Lemma}
We omit the proof since it follows as in \cite{fm2}. We observe only that, if $\lambda <0$, the thesis follows immediately by \cite[Lemma 4.3]{fm1} via Hypothesis \ref{Ass02}.4. Otherwise, if $\lambda > 0$, by definition of
$\varphi$ and by the assumption on $b$, one has
\[
\begin{aligned}
-s\lambda \int_{Q_T} \frac{\varphi_xb'}{b^2} w^2 dxdt &=
-s\lambda \int_{Q_T} \Theta \frac{\psi'b'}{b^2} w^2
dxdt\\&=-s \lambda d_1 \int_{Q_T} \Theta
\frac{(x-x_0)b'}{ab^2}e^{R(x-x_0)}w^2 dxdt\\
&\ge -s\lambda d_1K_2 \int_{Q_T} \frac{\Theta}{ab}e^{R(x-x_0)}w^2 dxdt.
\end{aligned}
\]
Since $w(t,\cdot)\in {\cal K}$ for every $t\in[0,1]$, for $w\in {\cal V}$, we get
\[
\int_{Q_T} \frac{\Theta}{ab}w^2 e^{R(x-x_0)}dxdt \le \int_{Q_T} \frac{\Theta}{ab}w^2dxdt\le 
C^* \int_{Q_T} \Theta (w_x)^2 dxdt.
\]
Hence,
\[
-s\lambda \int_{Q_T} \frac{\varphi_xb'}{b^2} w^2 dxdt \ge
-s\lambda d_1K_2 C^* \int_{Q_T} \Theta
(w_x)^2 dxdt \ge -sd_1K_2  \int_{Q_T} \Theta
(w_x)^2 dxdt.
\]
Again the thesis follows by  \cite[Lemma 4.3]{fm1}, as in \cite[Lemma 3.3]{fm2}.
\vspace{0.5cm}

From Lemma \ref{lemma1}, Lemma \ref{lemma4} and Lemma \ref{lemma2},
we deduce immediately that there exist two positive constants $C$
and $s_0$, such that for
all $s \ge s_0$,
\begin{equation}\label{D&BT1}
\begin{aligned}
\int_{Q_T}\frac{1}{a}L^+_s w L^-_s w dxdt &\ge Cs\int_{Q_T} \Theta
(w_x)^2 dxdt\\&+ Cs^3 \int_{Q_T}\Theta^3 \left(\frac{x-x_0}{a}\right)^2
w^2 dxdt\\
&-sd_1\int_0^T\Theta(t)\Big[(x-x_0)e^{R(x-x_0)^2}(w_x)^2\Big]_{x=0}^{x=1}dt.
\end{aligned}
\end{equation}

Thus, a straightforward consequence of \eqref{stimetta} and
\eqref{D&BT1} is the next result.
\begin{Lemma}
Assume Hypothesis $\ref{Ass02}$. Then, there exist two
positive constants $C$ (depending on $\lambda$) and $s_0$, such that for all $s \ge s_0$, \begin{equation}\label{2stelle}
\begin{aligned}
   & s\int_{Q_T} \Theta (w_x)^2dxdt  + s^3
\int_{Q_T}\Theta^3 \left(\frac{x-x_0}{a}\right)^2 w^2 dxdt\\&\le
    C\left(\int_{Q_T} h^2 \frac{e^{2s\varphi(t,x)}}{a}dxdt+sd_1\int_0^T\Theta(t)\Big[(x-x_0)e^{R(x-x_0)^2}(w_x)^2\Big]_{x=0}^{x=1}dt
    \right).
    \end{aligned}
    \end{equation}
\end{Lemma}
Recalling the definition of $w$, we have $v= e^{-s\varphi}w$ and
$v_{x}= -s\Theta \psi'e^{-s\varphi}w + e^{-s\varphi}w_{x}$. Thus,
substituting in \eqref{2stelle}, Theorem \ref{Cor1} follows if (Dbc) hold.

\subsection{Proof of Theorem  \ref{Cor1}  if (Nbc) hold} In this case we will proceed as in \cite{bfm1}, using Theorem \ref{Cor1} in the case of Dirichlet boundary conditions and  a technique based on cut off functions. 
\begin{proof}[Proof of Theorem \ref{Cor1}]
Choose $\alpha, \beta >0$ such that
$\alpha <\beta < x_0$, $1+\beta< 2-x_0$, and consider a
smooth function $\xi: [-1,2] \to \R$ such that $\xi \equiv 1$ in $[-\alpha,1+\alpha]$ and $ \xi \equiv 0$ in $ [-1, -\beta] \cup [1+\beta, 2]$.
Now, we consider
\begin{equation}\label{WN}
W(t,x):= \begin{cases}
v(t, 2-x), &x \in [1,2],\\
v(t,x), & x \in [0,1],\\
v(t,-x), & x \in [-1,0],
\end{cases}
\end{equation}
where $v$ solves 
\begin{equation}\label{P-adjoint-N}
\begin{cases}
v_t +a(x)v_{xx}  + \displaystyle \frac{\lambda}{b(x)}v=h(t,x)=h, & (t,x) \in Q_T,\\
v_x(t,0)=v_x(t,1)=0, &t \in (0,T).
\end{cases}
\end{equation}
Thus $W$ satisfies \eqref{PW}, where $\tilde a$, $\tilde b$ and $\tilde h$ are defined as in \eqref{tildea} and \eqref{tildeh}, respectively. Clearly, in this case $\tilde a$ and $\tilde b$ are $0$ at $x_0$ and, as before, $\tilde a, \tilde b$ belong to $W^{1,1}(-1,2)$ or  to $W^{1,\infty}(-1,2)$, if $a, b$ belong to  $W^{1,1}(0,1)$ or  to $W^{1,\infty}(0,1)$, respectively.
Now, set $Z:= \xi W$ and take $\delta >0$ such that $\beta +\delta<x_0$ and $1+\beta+ \delta < 2-x_0$. Clearly, $-x_0 < -\beta-\delta$. Then $Z$ solves
\[
\begin{cases}
Z_t + \tilde a Z_{xx} + \ds\lambda \frac{Z}{\tilde b }=H, & (t,x) \in (0,T) \times (-\beta-\delta,1+\beta+\delta),\\
Z(t,-\beta-\delta)=Z(t,1+\beta+\delta)=0, &  t \in (0,T),\\
\end{cases}
\]
 with $H:=\xi \tilde h+ \tilde a (\xi _{xx} W + 2\xi _x  W_x)$. Observe that $Z_x(t, -\beta-\delta)=Z_x(t, 1+\beta+\delta)=0$ and, by the assumption on $a$ and the fact that $\xi_x $, $\xi_{xx}$ are supported in $[-\beta,-\alpha]\cup[1+\alpha, 1+\beta]$, $H\in L^2((0,T); L^2_{\frac{1}{\tilde a}} (-\beta-\delta,1+\beta+\delta)).$
Now, define
$
\tilde \varphi(t,x) := \Theta(t) \tilde \psi (x),
$
where
\begin{equation}\label{tildepsi_nondiv}
\tilde \psi(x) := \begin{cases} \displaystyle\psi(2-x) =d_1\left[\int_{2-x_0}^x \frac{t-2+x_0}{\tilde
a(t)}e^{R(2-t-x_0)}dt -d_2\right], & x\in [1,2],\\
\psi(x), & x \in [0,1],\\
\displaystyle \psi(-x)= d_1\left[\int_{-x_0}^x \frac{t+x_0}{\tilde
a(t)}e^{R(-t-x_0)}dt -d_2\right], & x \in [-1,0].\end{cases}
\end{equation}
Thus, we can apply the
analogue of Theorem \ref{Cor1} with (Dbc) on $(-\beta-\delta,1+\beta+\delta)$ in place of
$(0,1)$ and with weight $\tilde \varphi$, obtaining that there exist
two positive constants $C$ and $s_0$ ($s_0$ sufficiently large), such that $Z$
satisfies, for all $s \ge s_0$,
\[
\begin{aligned}
&\int_0^T\int_{-\beta-\delta}^{1+\beta+\delta} \left(s\Theta (Z_x)^2 + s^3 \Theta^3
\left(\frac{x-x_0}{\tilde a}\right)^2 Z^2\right)e^{2s\tilde \varphi}dxdt\\
&\le C\left(\int_0^T\int_{-\beta-\delta}^{1+\beta+\delta}  H^{2}\frac{e^{2s\tilde \varphi}}{\tilde a}dxdt +
sd_1\int_0^T\left[ \Theta e^{R(x-x_0)^2}(x-x_0)(Z_x)^2
\right]_{x=-\beta-\delta}^{x=1+\beta+\delta}dt\right)\\
& = C\int_0^T\int_{-\beta-\delta}^{1+\beta+\delta}  H^{2}\frac{e^{2s\tilde \varphi}}{\tilde a}dxdt.
\end{aligned}
\]
By definition of $\xi$, $W$ and $Z$, we have
\[
\begin{aligned}
&\int_0^T\int_0^1 \left(s\Theta (v_x)^2 + s^3 \Theta^3
\left(\frac{x-x_0}{a} \right)^2v^2\right)e^{2s\varphi}dxdt
\\&=\int_0^T\int_0^1 \left(s\Theta (Z_x)^2 + s^3 \Theta^3
\left(\frac{x-x_0}{a} \right)^2Z^2\right)e^{2s\tilde \varphi}dxdt\\
& \le \int_0^T\int_{-\beta-\delta}^{1+\beta+\delta} \left(s\Theta (Z_x)^2 + s^3 \Theta^3
\left(\frac{x-x_0}{a} \right)^2 Z^2\right)e^{2s\tilde \varphi}dxdt\le C\int_0^T\int_{-\beta-\delta}^{1+\beta+\delta}  H^{2}\frac{e^{2s\tilde \varphi}}{\tilde a}dxdt.
\end{aligned}
\]
Using the fact that $\xi_x $ and $\xi_{xx}$ are supported in $[-\beta,-\alpha]\cup[1+\alpha, 1+\beta]$, it follows
\[
\begin{aligned}
&\int_0^T\int_{-\beta-\delta}^{1+\beta+\delta}  H^{2}\frac{e^{2s\tilde \varphi}}{\tilde a} dxdt = \int_0^T\int_{-\beta-\delta}^{1+\beta+\delta}(\xi \tilde h+ \tilde a (\xi _{xx} W + 2\xi _x  W_x))^2\frac{e^{2s\tilde \varphi}}{\tilde a}dxdt\\
&\le C\left( \int_0^T\int_{-\beta-\delta}^{1+\beta+\delta} \tilde h^2\frac{e^{2s\tilde \varphi}}{\tilde a}dxdt + \int_0^T\left(\int_{-\beta}^{-\alpha}+ \int_{1+\alpha}^{1+\beta}\right)(W^2 + W_x^2)e^{2s\tilde \varphi}dxdt \right)\\
&\le C \int_0^T \int_{-1}^2\left(
\frac{\tilde h^2}{\tilde a} + W_x^2+ W^2\right)e^{2s\tilde \varphi}dxdt.
\end{aligned}
\]
Hence, using the definitions of $\tilde \varphi$, $\tilde a$, $\tilde h$ and $W$, it results
\begin{equation}\label{sopra0}
\begin{aligned}
&\int_0^T\int_0^1 \left(s\Theta (v_x)^2 + s^3 \Theta^3
\left(\frac{x-x_0}{a}\right)^2 v^2\right)e^{2s\varphi}dxdt \le C \int_0^T \int_{-1}^2\left(
\frac{\tilde h^2}{\tilde a} + W_x^2+ W^2\right)e^{2s\tilde \varphi}dxdt\\
&\le C \int_0^T \int_0^1\left(
\frac{ h^2}{a} +v^2+  \Theta v_x^2\right)e^{2s \varphi}dxdt,
\end{aligned}
\end{equation}
for a positive constant $C$. Hence, we can choose $s_0$ so large that, for all $s \ge s_0$,
\[
\begin{aligned}
&\int_0^T\int_0^1 \left(s\Theta (v_x)^2 + s^3 \Theta^3
\left(\frac{x-x_0}{a}\right)^2 v^2\right)e^{2s\varphi}dxdt
\\&\le C \left(\int_0^T  \int_0^1
h^2 \frac{e^{2s \varphi}}{a}dxdt+ \int_0^T \int_0^1 v^2 e^{2s\varphi} dxdt\right),
\end{aligned}
\]
for a positive constant $C$. 

\vspace{0.3cm}
Assume now that $\omega$ is a strict subset of $(0,1)$ such that $x_0 \in \omega$. Then, by \eqref{carN},
\[
\begin{aligned}
&\int_{Q_T} \left(s\Theta (v_x)^2 + s^3 \Theta^3
\left(\frac{x-x_0}{a} \right)^2v^2\right)e^{2s\varphi}dxdt\\
& \le C\left(\int_{Q_T} h^{2}\frac{e^{2s\varphi}}{a}dxdt + \int_{Q_T} v^{2}e^{2s\varphi}dxdt\right)\\
& =  C \left(\int_{Q_T}
 h^2\frac{e^{2s\varphi}}{a} dxdt +  \int_0^T \int_{(0,1) \setminus \omega} v^2e^{2s  \varphi}dxdt + \int_0^T \int_{\omega} v^2e^{2s  \varphi}dxdt \right)\\
 & \le C \left(\int_{Q_T} h^{2}\frac{e^{2s\varphi}}{a}dxdt +\int_0^T \int_{(0,1) \setminus \omega}  \Theta^3\left(\frac{x-x_0}{a}\right)^2 v^2e^{2s  \varphi}dxdt + \int_0^T \int_{\omega}v^2e^{2s  \varphi}dxdt \right)\\
& \le C \left(\int_{Q_T} h^{2}\frac{e^{2s\varphi}}{a}dxdt +\int_{Q_T}  \Theta^3\left(\frac{x-x_0}{a}\right)^2 v^2e^{2s  \varphi}dxdt + \int_0^T \int_{\omega}v^2e^{2s  \varphi}dxdt \right).\\
\end{aligned}
\]
Hence, we can choose $s_0$ so large that, for all $s \ge s_0$ and for a positive constant $C$:
\[
\int_{Q_T} \left(s\Theta (v_x)^2 + s^3 \Theta^3
\left(\frac{x-x_0}{a}\right)^2 v^2\right)e^{2s\varphi}dxdt\\
\le C \left(\int_{Q_T} h^{2}\frac{e^{2s\varphi}}{a}dxdt  + \int_0^T \int_{\omega}v^2e^{2s  \varphi}dxdt \right).
\]
\end{proof}

Observe that \eqref{carN} and \eqref{carN_1} are the analogous estimates proved in \cite{bfm1} when $\lambda=0$.

\section{Applications to observability inequality}\label{sec5}
In this section we shall apply the just established Carleman inequalities to observability issues.
For this, we assume that the control set $\omega$ satisfies
the following assumption:
\begin{Assumptions} \label{ipotesiomega}
The subset $\omega$ is such that
\begin{itemize}
\item it is an interval containing the degeneracy point:
\begin{equation}\label{omega1}
\omega=(\alpha,\beta) \subset (0,1) \mbox{ is such that $x_0 \in
\omega$},
\end{equation}
or
\item it is an interval lying on one
side of the degeneracy point:
\begin{equation}\label{omega}
\omega=(\alpha,\beta) \subset (0,1) \mbox{ is such that $x_0\not \in
\bar \omega$}.
\end{equation}
\end{itemize}
\end{Assumptions}

On the functions $a$, $b$ and on the constant $\lambda$ we make the following
assumptions:
\begin{Assumptions}\label{Ass04} Hypothesis $\ref{Ass02}$ is satisfied. Moreover,
if Hypothesis $\ref{Ass0}$ or $\ref{Ass0_1}$ holds,  there exist
two functions $\fg \in L^\infty_{\rm loc}([0,1]\setminus \{x_0\})$, $\fh \in W^{1,\infty}_{\rm loc}([0,1]\setminus \{x_0\};L^\infty(0,1))$ and
two strictly positive constants $\fg_0$, $\fh_0$ such that $\fg(x) \ge \fg_0$ and
\begin{equation}\label{5.3'}
\frac{a'(x)}{2\sqrt{a(x)}}\left(\int_x^B\fg(t) dt + \fh_0 \right)+ \sqrt{a(x)}\fg(x) =\fh(x,B)\quad \text{for a.e.} \; x,B \in [0,1]
\end{equation}
with $x<B<x_0$ or $x_0<x<B$.
\end{Assumptions}

\begin{Assumptions}\label{ipa}
If $x_0 \not \in \omega$, (Nbc) hold and $K_1+ K_2 <1$, then
\[
\max_{[0,1]}a < \ds \frac{1}{C_{HP,1}},
\]
where $C_{HP,1}$ is the Hardy--Poincar\'e constant of Corollary \ref{HPN1}.
\end{Assumptions}

\begin{Remark}
Since we require identity \eqref{5.3'} far from $x_0$, once $a$ is given, it is easy to find $\fg,\fh,\fg_0$ and $\fh_0$ with the desired properties (see \cite[Remark 4]{bfm1} for some examples).
\end{Remark}

Now, we associate to problem \eqref{linear} the homogeneous adjoint problem
\begin{equation}\label{h=0}
\begin{cases}
v_t +av_{xx} +\displaystyle \frac{ \lambda}{b(x)}v= 0, &(t,x) \in
Q_T,
\\[5pt]
Bv(0)=Bv(1) =0, & t \in (0,T),
\\[5pt]
v(T,x)= v_T(x),
\end{cases}
\end{equation}
where $T>0$ is given and $v_T(x) \in L^2_{\frac{1}{a}}(0,1)$. By the Carleman estimates given in Theorem
\ref{Cor1}, we will deduce the following observability inequality
for all the degenerate cases:
\begin{Proposition}
\label{obser.} Assume Hypotheses $\ref{ipotesiomega}$ -- $\ref{ipa}$.
  There exists a positive constant $C_T$ such that every solution $v
\in  \mathcal U$ of
\eqref{h=0} satisfies
 \begin{equation}\label{obser1.}
\int_0^1v^2(0,x)\frac{1}{a} dx \le C_T\int_0^T \int_{\omega}v^2(t,x)\frac{1}{a}dxdt,
\end{equation}
where
\begin{equation}\label{U}
\mathcal U:= C\big([0, T]; L^2_{\frac{1}{a}}(0,1)\big) \cap L^2 (0,T; \cK).
\end{equation}
\end{Proposition}

\subsection{Proof of Proposition \ref{obser.}}
We will give some preliminary
results. As a first step, consider the adjoint problem
\[
(P_i) \quad
\begin{cases}
v_t +A_iv= 0, &(t,x) \in Q_T,
\\[5pt]
Bv(0)=Bv(1) =0, & t \in (0,T),
\\[5pt]
v(T,x)= v_T(x) \,\in D({A_i}^2),
\end{cases}
\]
where
\[
D({A_i}^2) = \Big\{u \,\in \,D({ A_i })\;\big|\; { A_i }u \,\in \,D({ A_i })
\;\Big\}.
\]
Observe
that $D({ A _i}^2)$ is densely defined in $D({ A_i})$ (see, for
example, \cite[Lemma 7.2]{b}) and hence in $L^2_{\frac{1}{a}}(0,1)$. As in
\cite{fm1}, define 
\[
\cal Q:=\Big\{ v\text{ is a solution of } (P_i) \Big\}. \]
Obviously (see, for example, \cite[Theorem 7.5]{b}),
\[ \cal{Q}\subset
\mathcal S \subset \mathcal{V} \subset
\cal{U},
\]
where $\mathcal{V}$  and $\mathcal U$ are defined in \eqref{v1} and \eqref{U}, respectively, and
\[
\mathcal{S}:=C^1\big([0,T]\:;\cW \big).
\]
In order to prove Proposition \ref{obser.}, we need the following result:
\begin{Lemma}\label{lemma3}
Assume Hypotheses $\ref{ipotesiomega}$ and $\ref{Ass04}$. Then there exist two positive
constants $C$ and $s_0$ such that every solution $v \in \cal Q$ of
$(P_i)$, $i=1,2$, satisfies, for all $s \ge s_0$,
\[
\int_{Q_T}\left( s \Theta  (v_x)^{2} + s^3 \Theta ^3
\left(\frac{x-x_0}{a}\right)^2 v^{2}\right) e^{{2s\varphi}}  dxdt\le C
\int_0^T\int_{\omega}\frac{v^{2}}{a} dxdt.
\]
Here $\Theta$ and $\varphi$ are as before.
    \end{Lemma}

The proof of the previous lemma follows by the next Caccioppoli's inequality:

\begin{Proposition}[Caccioppoli's inequality]\label{caccio1}
Assume Hypothesis $\ref{Ass03_new}$ and \eqref{5.3'} if Hypothesis $\ref{Ass0}$ or $\ref{Ass0_1}$ holds.
Let $\omega'$ and $\omega$ two open subintervals of $(0,1)$ such
that $\omega'\subset \subset \omega \subset  (0,1)$ and $x_0 \not
\in \overline{\omega}$. Let $\varphi(t,x)=\Theta(t)\Upsilon(x)$, where
$\Theta$ is defined in \eqref{theta} and
\[
\Upsilon \in C([0,1],(-\infty,0))\cap
C^1([0,1]\setminus\{x_0\},(-\infty,0))
\]
satisfies
\begin{equation}\label{stimayx}
|\Upsilon_x|\leq \frac{c}{\sqrt{a}} \mbox{ in }[0,1]\setminus\{x_0\}.
\end{equation}
Then there exist two positive constants
$C$ and $s_0$ such that every solution $v \in \cal Q$ of the
adjoint problem $(P_i)$, $i=1,2$, satisfies
\begin{equation}\label{Caccio}
   \int_{0}^T \int _{\omega'}   (v_x)^2e^{2s\varphi } dxdt
    \ \leq \ C \int_{0}^T \int _{\omega}   v^2  dxdt  \ \leq \ C \int_{0}^T \int _{\omega}   v^2  \frac{1}{a}dxdt 
\end{equation}
for all $s\geq s_0$.
\end{Proposition}

See \cite[Remark 10]{fm1} for some comments on \eqref{stimayx}.

\begin{proof}[Proof of Proposition \ref{caccio1}]
The proof is an adaptation of the one of \cite[Proposition 5.4]{fm1},
so we will skip some details. Let us consider a smooth function
$\xi: [0,1] \to \Bbb R$ such that
\[
\begin{cases}
0 \leq \xi (x)  \leq 1, &  \text{for all } x \in [0,1], \\
\xi (x) = 1 ,  &   x \in \omega', \\
\xi (x)=0, &     x \in [\, 0, 1 ]\setminus \omega.
\end{cases}
\]
Hence, by definition of $\vp$, we have
\[
    \begin{aligned}
    0 &= \int _0 ^T \frac{d}{dt} \left(\int _0 ^1 \xi ^2 e^{2s\varphi}
    v^2dx\right)dt
    = \int_{Q_T}\big(2s \xi ^2  \varphi _t e^{2s\varphi} v^2 + 2 \xi ^2
    e^{2s\varphi} vv_t \big)dxdt\\
    & \mbox{(since $v$ solves $(P_i)$, $i=1,2$)}
   \\
    &= 2\int_{Q_T} s\xi^2 \varphi _t e^{2s\varphi} v^2 dxdt+  2\int_{Q_T}( \xi
    ^2e^{2s\varphi}a)_xvv_xdxdt\\
& +  2\int_{Q_T}  \xi
    ^2e^{2s\varphi}a(v_x)^{2}dxdt- 2\lambda \int_{Q_T}\xi^2e^{2s\varphi}\frac{v^2}{b}dxdt.
\end{aligned}
    \]
    
    If $\lambda \le 0$, one has
    \[
    \begin{aligned}
   2\int_{Q_T}  \xi
    ^2e^{2s\varphi}a(v_x)^{2}dxdt &= - 2\int_{Q_T} s\xi^2 \varphi _t e^{2s\varphi} v^2 dxdt-  2\int_{Q_T}( \xi
    ^2e^{2s\varphi}a)_xvv_xdxdt
    \\
    &+  2\lambda \int_{Q_T}\xi^2e^{2s\varphi}\frac{v^2}{b}dxdt\\
    & \le - 2\int_{Q_T} s\xi^2 \varphi _t e^{2s\varphi} v^2 dxdt-  2\int_{Q_T}( \xi
    ^2e^{2s\varphi}a)_xvv_xdxdt.
    \end{aligned}
    \]
    Hence, proceeding as in \cite[Proposition 5.4]{fm1}, the claim follows.

    If $\lambda >0$, we can apply Lemmas \ref{L2}, \ref{L2''}, if (Dbc) hold, or Lemmas \ref{L2'}, \ref{L2'''},  if (Nbc) are in force, to $w= \xi \sqrt{a} e^{s\varphi}v$. Hence, fixed $\epsilon >0$, by the Cauchy-Schwarz inequality and by definition of $\xi$, we get,  for some $C_\epsilon >0$,
    \[
    \begin{aligned}
   2\lambda \int_{Q_T}\xi^2e^{2s\varphi}\frac{v^2}{b}dxdt &= 2\lambda \int_{Q_T}\frac{w^2}{ab}dxdt \le 2\lambda C^*\int_{Q_T}(w_x)^2dxdt
  \\
  & \le4 \lambda C^*\left(\epsilon \int_0^T\int_\omega \xi^2 e^{2s\varphi} a (v_x)^2dxdt+ C_\epsilon \int_0^T\int_\omega [(\xi e^{s\varphi}\sqrt{a})_x]^2 v^2 dxdt\right)=:J_1,
   \end{aligned}
    \]
    if (Dbc) hold or if (Nbc) are in force and Lemma \ref{L2'''} is applied, and
    \[
    \begin{aligned}
     2\lambda \int_{Q_T}\xi^2e^{2s\varphi}\frac{v^2}{b}dxdt &= 2\lambda \int_{Q_T}\frac{w^2}{ab} \le 2\lambda C^* \left[\int_{Q_T} (w_x)^2dxdt + \int_{Q_T} w^2dxdt\right]\\
     &\left. \begin{aligned}
     &\le4\lambda C^* \left(\epsilon\int_0^T\int_\omega \xi^2 e^{2s\varphi} a (v_x)^2dxdt+C_\epsilon \int_0^T\int_\omega [(\xi e^{s\varphi}\sqrt{a})_x]^2 v^2 dxdt \right)
     \\
     &+2\lambda C^*\int_0^T\int_\omega \xi^2 e^{2s\varphi} av^2dxdt
     \end{aligned} \right\}=:J_2,
    \end{aligned}
    \]
      in the case of (Nbc) and Lemma \ref{L2'}. In every case, setting $J:= J_1$ or $J:=J_2$, we have
      \begin{equation}\label{I}
      \begin{aligned}
       2\int_0^T \int_{\omega}\xi^2 e^{2s\varphi} a (v_x)^2dxdt
&\le- 2 \int_0^T \int_{\omega}s \xi ^2 \varphi _t e^{2s\varphi} v^2dxdt
\\
&  -2\int_0^T \int_{\omega} \left(\xi^2e^{2s\varphi}a\right)_xvv_x\:dxdt + J.
      \end{aligned}
      \end{equation}
      Now, as in \cite[Proposition 5.4]{fm1},
      \[
   -2\int_0^T \int_{\omega} \left(\xi^2e^{2s\varphi}a\right)_xvv_x\:dxdt \le \int_0^T \int_{\omega} \xi ^2 e^{2s\varphi} a (v_x)^2 dxdt
 + 4\int_0^T \int_{\omega}[\left(\xi
e^{s\vp}\sqrt{a}\right)_x]^2v^2dxdt.
      \]
      Substituting this last inequality in \eqref{I} and using the definition of $J$, it follows
      \[
      \begin{aligned}
      \begin{aligned}
  &     \int_0^T \int_{\omega}\xi^2 e^{2s\varphi} a (v_x)^2dxdt
\le - 2 \int_0^T \int_{\omega}s \xi ^2 \varphi _t e^{2s\varphi} v^2dxdt
+ 4\int_0^T \int_{\omega}[\left(\xi
e^{s\vp}\sqrt{a}\right)_x]^2v^2dxdt  + J\\
& \le - 2 \int_0^T \int_{\omega}s \xi ^2 \varphi _t e^{2s\varphi} v^2dxdt
+ 4(1+\lambda C^*C_\epsilon)\int_0^T \int_{\omega}[\left(\xi
e^{s\vp}\sqrt{a}\right)_x]^2v^2dxdt \\
&+ 4\lambda C^* \epsilon \int_0^T \int_{\omega}\xi^2 e^{2s\varphi} a (v_x)^2dxdt+ 2\lambda C^*\int_0^T \int_{\omega}\xi^2 e^{2s\varphi} a v^2dxdt.
      \end{aligned}
      \end{aligned}
      \]
     Hence
\begin{equation}\label{I_1}
    \begin{aligned}
    &(1-4\lambda C^* \epsilon) \int_0^T \int_{\omega}\xi^2 e^{2s\varphi} a (v_x)^2dxdt\le - 2
\int_0^T \int_{\omega}s \xi^2 \varphi _t e^{2s\varphi} v^2dxdt
    \\&+ 4(1+\lambda C^* C_\epsilon )\int_0^T \int_{\omega}[\left(\xi
e^{s\varphi}\sqrt{a}\right)_x]^2v^2dxdt+2\lambda C^*\int_0^T \int_{\omega}\xi^2 e^{2s\varphi} a v^2dxdt.
\end{aligned}
\end{equation}
Moreover, using the fact that $x_0 \not \in \overline{\omega}$, we have, as in \cite[Proposition 5.4]{fm1}, the existence of a
positive constant $C$ depending on $\epsilon$ such that
\[
\begin{aligned}
- 2 &\int_0^T \int_{\omega}s \xi^2 \varphi_t e^{2s\varphi}
v^2dxdt
 + 4(1+\lambda C^* C_\epsilon)\int_0^T \int_{\omega}[\left(\xi
e^{s\vp}\sqrt{a}\right)_x]^2v^2dxdt\\
 &  \le C\int_0 ^T \int _{\omega} v^2 dxdt.
   \end{aligned}
\]
Substituting in \eqref{I_1}, we get
\[
\begin{aligned}
&(1-4\lambda C^*\epsilon) \int_0^T \int_{\omega}\xi^2 e^{2s\varphi} a (v_x)^2dxdt\le C \int_0 ^T \int _{\omega} v^2dxdt+2\lambda C^*\int_0^T \int_{\omega}\xi^2 e^{2s\varphi} a v^2dxdt\\
&\le C \int_0 ^T \int _{\omega} v^2dxdt \le  C\int_0 ^T \int _{\omega} v^2 \frac{1}{a}dxdt,
\end{aligned}
\]
for a positive constant $C$ (still depending on $\epsilon$).
Since $x_0 \not \in \overline{\omega'}$
and choosing $\ds\epsilon < \frac{1}{4\lambda C^*}$, we can prove that there exists a positive constant $C$ such that
\[
\begin{aligned}
&\inf_{\omega '}a(x)\int_0^{T}\int _{\omega '} e^{2s\varphi}
(v_x)^2dxdt \le \int_0^T \int_{\overline{\omega  '}} \xi^2
e^{2s\varphi} a (v_x)^2dxdt\\
&\le \int_0^T \int_{\omega}
\xi ^2 e^{2s\varphi} a (v_x)^2dxdt\le C\int_0 ^T \int _{\omega} v^2 \frac{1}{a}dxdt.  \end{aligned}
\]
\end{proof}
\begin{Remark}\label{remCaccio}
Clearly \eqref{Caccio} holds also if the state space $(0,1)$ does not contain a degenerate point.
\end{Remark}

\begin{proof}[Proof of Lemma \ref{lemma3} if $x_0 \not \in \overline \omega$] 
Recall that $\omega =(\alpha, \beta)$ and suppose $x_0 <\alpha$ (the proof is similar if we assume that
$\beta < x_0$ with simple adaptations); moreover, set $\tau:=
\ds \frac{2\alpha +\beta}{3}$ and $\gamma:= \ds \frac{\alpha +2\beta}{3}$,
so that $\alpha<\tau<\gamma<\beta$. Now, fix $\tilde \alpha \in (\alpha, \tau)$, $\tilde \beta \in (\gamma , \beta)$ and consider a smooth
function $\xi: [0,1] \to \Bbb R$ such that
\[\begin{cases}
    0 \leq \xi (x)  \leq 1, &  \text{ for all } x \in [0,1], \\
    \xi (x) = 1 ,  &   x \in [\tau, \gamma],\\
    \xi (x)=0, &     x \in [0,1]\setminus (\tilde \alpha, \tilde \beta).
    \end{cases}
\]
Define $w:= \xi v$, where $v$ is any fixed solution of
$(P_i)$, $i=1,2$. Hence, neglecting the final--time datum (of no
interest in this context), $w$ satisfies
\[
\begin{cases}
w_t + a  w_{xx} +\ds \lambda\frac{w}{b}= a (\xi _{xx} v + 2\xi _x  v_x )=:f,&
(t,x) \in Q_T, \\
w(t,0)= w(t, 1)=0, & t \in (0,T).
\end{cases}
\]
Applying Theorem \ref{Cor1} with (Dbc), there exists two positive constants $C$ and $s_0$ such that, for all $s \ge s_0$,
\begin{equation}\label{car91}
\begin{aligned}
\int_{Q_T}\Big( s \Theta   (w_x)^2 + s^3 \Theta^3
   \left(\frac{x-x_0}{a}\right)^2\ w^2 \Big)
    e^{2s \varphi} \, dx dt
       \le C\int_{Q_T}\frac{e^{2s \varphi}}{a} f^2  dxdt.
\end{aligned}
\end{equation}
Then, using the definition of $\xi$  and in
particular the fact that  $\xi_x$ and  $\xi_{xx}$ are supported in
$\hat \omega$, where $\hat \omega:= (\tilde\alpha, \tau) \cup( \gamma,
\tilde\beta) $, we can write
\begin{equation}\label{fa}
\frac{f^2}{a}= a(\xi _{xx} v +2 \xi _x v_x)^2 \le C( v^2+
(v_x)^2)\chi_{\hat \omega}.
\end{equation}
Hence, using the fact that 
$\hat \omega\subset \subset \omega$ and $x_0\not \in \overline \omega$, we find
\begin{equation}\label{stimacar1}
\begin{aligned}
&\int_0^T\int_{\tau}^{\gamma}\left( s \Theta (v_x)^{2} + s^3
\Theta ^3 \left(\frac{x-x_0}{a}\right)^2 v^{2}\right)
e^{{2s\varphi}} dxdt\\
&=\int_0^T\int_{\tau}^{\gamma}\Big( s \Theta   (w_x)^2 + s^3
\Theta^3 \left(\frac{x-x_0}{a}\right)^2 w^2 \Big)
e^{2s \varphi} \, dx dt\\
&\le\int_{Q_T}\Big( s \Theta  (w_x)^2 + s^3 \Theta^3
\left(\frac{x-x_0}{a} \right)^2w^2 \Big)
e^{2s \varphi} \, dx dt\\
& \mbox{ (by \eqref{car91} and \eqref{fa})}\\
&\le C  \int_0^T \int_{\hat \omega}e^{2s \varphi}( v^2+
(v_x)^2)dxdt\\
& \mbox{ (by Proposition \ref{caccio1} with $\vp=\Theta\psi$ and using the fact that}\\
&\mbox{ $e^{2s\varphi}$ is bounded)}\\
&\le C \int_0^T \int_{\omega} \frac{v^2}{a} dxdt.
\end{aligned}
\end{equation}
Now, consider a smooth function $\eta: [0,1] \to \Bbb R$ such that
\[
\begin{cases}
    0 \leq \eta (x)  \leq 1, &  \text{ for all } x \in [0,1], \\
   \eta (x) = 1 ,  &   x \in [\gamma , 1],\\
  \eta (x)=0, &     x \in \left[0,\tau\right],
\end{cases}
\]
and define $z:= \eta v$. Then $z$  satisfies the non degenerate problem
\begin{equation}\label{eq-z*}
\begin{cases}
z_t + a  z_{xx}+ \lambda \ds\frac{z}{b}= h,  &(t,x) \in(0,T)\times (\alpha,1),\\
Bz(\alpha)= Bz(1)=0, & t \in (0,T),
\end{cases}
\end{equation}
with $h:=a (\eta _{xx} v + 2\eta _x  v_x)\in L^2\big((0,T)\times
(\alpha, 1)\big)$.

Moreover, since the problem is {\em non degenerate},
we can apply, thanks to Remark \ref{rembo}, Proposition \ref{classical Carleman} in  $(\alpha, 1)$. In the following of the proof, we will distinguish between  Dirichlet boundary conditions and Neumann ones.

{\it\underline{Dirichlet boundary conditions:}} By Proposition \ref{classical Carleman}  there exist two positive constants $C$ and $s_0$ such that
\begin{equation}\label{D}
\int_0^T\int_\alpha^1 \left(s\Theta (z_x)^2 + s^3 \Theta^3
 z^2\right)e^{2s\Phi}dxdt\le C\int_0^T\int_\alpha^1 \frac{h^{2}}{a}e^{2s\Phi}dxdt,
\end{equation}
for $s\geq s_0$. Observe that
the boundary terms in \eqref{carcorretta} are non positive ($z_x(t,0)=0$ for all $t\in [0,T]$).

Now, proceeding as in \eqref{fa}, we get that there exists a positive constant $C$ such that
$\ds
 \frac{h^2}{a} \le  C (v^2+ v_x^2)\chi_{\tilde\omega},
$ where $\tilde \omega =(\tau, \gamma)$.
Hence, by Remark \ref{remCaccio}, we can apply  Proposition
\ref{caccio1}, and recalling what the support of $\eta$ is, we get
\begin{equation}\label{standa}
\begin{aligned}
&\int_0^T\int_\alpha^1 \left(s\Theta (z_x)^2 + s^3 \Theta^3
 z^2\right)e^{2s\Phi}dxdt \\
&\le C \int_0^T\int_{\tilde \omega}e^{2s\Phi}( v^2+ (v_x)^2)dxdt\le
C \int_0^T\int_{\tilde \omega} v^2dxdt+ C\int_0^T\int_{\tilde
\omega}e^{2s\Phi}(v_x)^2dxdt \\& \le C \int_0^T\int_{\omega}\frac{
v^2}{a}dxdt.
\end{aligned}
\end{equation}
Since $x_0 \not \in (\alpha, 1)$, one has that there exists $k>0$ such that
\begin{equation}\label{5.16'}
\begin{aligned}
&\int_0^T\int_{\alpha}^1 \Big(s \Theta (z_x)^2
 + s^3 \Theta^3
     \left( \frac{x-x_0}{a}\right)^2z^2\Big) e^{2s\varphi}dxdt \\
& \le k \int_0^T\int_{\alpha}^1 s \Theta (z_x)^2e^{2s\Phi} dxdt +k\int_0^T\int_{\alpha}^1 s^3\Theta^3 z^2e^{2s\Phi} dxdt\\
& \le C \int_0^T \int_{\omega} \frac{v^2}{a}dxdt,
\end{aligned}
\end{equation}
for a positive constant $C$ and $s$ large enough. Hence, by
definition of $z$ and by the inequality above, we get
\begin{equation}\label{stimacar21}
\begin{aligned}
&\int_0^T\int_{\gamma}^1 \Big(s \Theta  (v_x)^2 + s^3 \Theta^3
\left(\frac{x-x_0}{a}\right)^2v^2\Big) e^{2s\varphi} dxdt\\
&=\int_0^T\int_{\gamma}^1 \Big(s \Theta  (z_x)^2 + s^3 \Theta^3
\left(\frac{x-x_0}{a}\right)^2z^2\Big) e^{2s\varphi}dxdt\\
& \le \int_0^T\int_{\alpha}^1 \Big(s \Theta  (z_x)^2 + s^3 \Theta^3
\left( \frac{x-x_0}{a}\right)^2z^2\Big) e^{2s\varphi}dxdt
\\
&\le C \int_0^T \int_{\omega} \frac{v^2}{a}dxdt,
\end{aligned}
\end{equation}
for a positive constant $C$ and for $s$ large enough.
Thus,  there exists two positive constants $C$ and $s_0$ such that, by \eqref{stimacar1} and \eqref{stimacar21},
\begin{equation}\label{carin01}
\begin{aligned}
\int_0^T \int _{\tau}^{1}  \Big( s \Theta  (v_x)^2 + s^3 \Theta^3
\left(\frac{x-x_0}{a}\right)^2  v^2 \Big) e^{2s \varphi } \, dx dt
\le C \int_{0}^T \int _{\omega}   \frac{v^2 }{a} dxdt,
\end{aligned}
\end{equation}
for all $s \ge s_0$.
To complete the
proof it is sufficient to prove a similar inequality for
$x\in[0,\tau]$. To this aim, we follow a reflection procedure
already introduced in \cite{fm} considering $W$ given by
\[
W(t,x):= \begin{cases} v(t,x), & x \in [0,1],\\
-v(t,-x), & x \in [-1,0]
\end{cases}
\]
and the functions $\tilde a$ and $\tilde b$ introduced in \eqref{tildea} but restricted to $[-1,1]$, i.e.
\[
\tilde a(x):= \begin{cases} a(x), & x \in [0,1],\\
a(-x), & x \in [-1,0]
\end{cases}\quad\text{and} \quad\tilde b(x):= \begin{cases} b(x), & x \in [0,1],\\
b(-x), & x \in [-1,0],
\end{cases}
\]
so that $W$ satisfies the problem
\[
\begin{cases}
W_t +\tilde a W_{xx}+ \lambda \ds\frac{W}{\tilde b}= 0, &(t,x) \in  (0,T)\times (-1,1),\\
W(t,-1)=W(t,1) =0, & t \in (0,T).
\end{cases}
\]
Now, consider a cut off function $\rho: [-1,1] \to \Bbb R$ such that
\[\begin{cases}
0 \leq \rho (x)  \leq 1, &  \text{ for all } x \in [-1,1], \\
\rho (x) = 1 ,  &   x \in [-\tau, \tau],\\
    %\left[\frac{2\lambda_1+\beta_1}{3}, \frac{\lambda_2+ 2 \beta_2}{3} \right], \\
\rho (x)=0, &     x \in \left[-1,-\gamma\right ]\cup
\left[\gamma,1\right],
\end{cases}
\]
and define $Z:= \rho W$. Then $Z$ satisfies
\begin{equation}\label{eq-Z*}
\begin{cases}
Z_t + \tilde aZ_{xx}+ \lambda \ds \frac{Z}{\tilde b}=\tilde h,  &(t,x) \in (0,T)\times (-1,1),\\
Z(t,-1)= Z(t,1)=0, & t \in (0,T),
\end{cases}
\end{equation}
where $\tilde h=\tilde a(\rho_{xx}W+2\tilde \rho_xW_x)$. Now, considering the function $\tilde \varphi$ introduced in \eqref{tildepsi_nondiv} but restricted to $[-1,1]$, i.e. $\tilde \varphi(t,x) := \Theta(t) \tilde \psi (x)$ with
\begin{equation}\label{tildepsi}
\tilde \psi(x) := \begin{cases}
\psi(x), & x \ge 0,\\
\displaystyle \psi(-x)= d_1\left[\int_{-x_0}^x \frac{t+x_0}{\tilde
a(t)}e^{R(t+x_0)^2}dt-d_2\right], & x <0,\end{cases}
\end{equation}
we use the analogue of Theorem \ref{Cor1} on $(- 1,
1)$ in place of $(0,1)$ and with $\varphi$ replaced by $\tilde
\varphi$. Moreover, using the fact that $Z_x(t, -1)=Z_x(t,
1)=0$, the definition of $W$ and the fact that $\rho$ is
supported in $\left[-\gamma,-\tau\right] \cup\left[\tau,
\gamma\right]$, we get
\begin{equation}\label{*}
\begin{aligned}
& \int_0^T\int_{-1}^1  \left(s\Theta (Z_x)^2 + s^3
\Theta^3
\left(\frac{x-x_0}{\tilde a} \right)^2Z^2\right)e^{2s\tilde\varphi}dxdt\\
& \le c
\int_0^T\int_{-1}^1  \tilde h^{2}\frac{e^{2s\tilde\varphi}}{\tilde a}dxdt\\
&\le C \int_0^T \int_{-\gamma}^{-\tau}(
W^2+ (W_x)^2)e^{2s \tilde \varphi}dxdt + C\int_0^T \int_{\tau}^{\gamma}(W^2+ (W_x)^2)e^{2s\varphi}dxdt\\
&\mbox{(since $\tilde \psi (x)= \psi (-x)$, for $x <0$)}\\
&= 2C\int_0^T \int_{\tau}^{\gamma}(W^2+ (W_x)^2)e^{2s\varphi}dxdt
= 2C\int_0^T \int_{\tau}^{\gamma}(v^2+
(v_x)^2)e^{2s\varphi}dxdt\\
& \mbox{ (by Propositions \ref{caccio1}) }\\
& \le C \int_0^T \int_{\omega} \frac{v^2}{a}dxdt,
\end{aligned}
\end{equation}
for some positive constants $c$, $C$ and $s$ large enough.
 Hence, by definitions of $Z$, $W$ and $\rho$, and using the previous inequality one has
\begin{equation}\label{car101}
\begin{aligned}
&\int_0^T\int_{0}^{\tau}  \left(s\Theta (v_x)^2 + s^3 \Theta^3
\left(\frac{x-x_0}{a} \right)^2v^2\right)e^{2s\varphi}dxdt\\
&= \int_0^T\int_{0}^{\tau}  \left(s\Theta (W_x)^2 + s^3 \Theta^3
\left(\frac{x-x_0}{a} \right)^2W^2\right)e^{2s\varphi}dxdt\\
&= \int_0^T\int_{0}^{\tau}  \left(s\Theta (Z_x)^2 + s^3 \Theta^3
\left(\frac{x-x_0}{a} \right)^2Z^2\right)e^{2s\varphi}dxdt\\
&\le \int_0^T\int_{-1}^1   \left(s\Theta (Z_x)^2 + s^3
\Theta^3
\left(\frac{x-x_0}{a} \right)^2Z^2\right)e^{2s\tilde \varphi}dxdt\\
&\le C \int_0^T \int_{\omega} \frac{v^2}{a}dxdt,
\end{aligned}
\end{equation}
for a positive constant $C$ and $s$ large enough. Therefore, by
\eqref{carin01} and \eqref{car101}, the conclusion follows.

{\it \underline{Neumann boundary conditions:}} In this case \eqref{D} becomes
\begin{equation}\label{N}
\int_0^T\int_\alpha^1 \left(s\Theta (z_x)^2 + s^3 \Theta^3
 z^2\right)e^{2s\Phi}dxdt\le C\left(\int_0^T\int_\alpha^1 \frac{h^{2}}{a}e^{2s\Phi}dxdt + \int_0^T\int_{\tilde \omega}  z^2e^{2s\Phi}dxdt  \right),
\end{equation}
 for a positive constant $C$ and for all $s\geq s_0$. Here, we recall, $\tilde \omega =(\tau, \gamma)$.
As for \eqref{standa}, we get
\begin{equation}\label{standaN}
\begin{aligned}
&\int_0^T\int_\alpha^1 \left(s\Theta (z_x)^2 + s^3 \Theta^3
 z^2\right)e^{2s\Phi}dxdt \\
&\le C \int_0^T\int_{\tilde \omega}e^{2s\Phi}( v^2+ (v_x)^2)dxdt+ C\int_0^T\int_{\tilde \omega} z^2e^{2s\Phi}dxdt\\
&(\text{since } z = \eta v)\\
&\le
C \int_0^T\int_{\tilde \omega} v^2dxdt+ C\int_0^T\int_{\tilde
\omega}e^{2s\Phi}(v_x)^2dxdt \\& \le C \int_0^T\int_{\omega}\frac{
v^2}{a}dxdt.
\end{aligned}
\end{equation}
Proceeding as before (see \eqref{5.16'} and \eqref{stimacar21}), there exists $k>0$ such that
\begin{equation}\label{stimacar21N}
\begin{aligned}
&\int_0^T\int_{\gamma}^1 \Big(s \Theta  (v_x)^2 + s^3 \Theta^3
\left(\frac{x-x_0}{a}\right)^2v^2\Big) e^{2s\varphi} dxdt\\
&\le
\int_0^T\int_{\alpha}^1 \Big(s \Theta (z_x)^2
 + s^3 \Theta^3
     \left( \frac{x-x_0}{a}\right)^2z^2\Big) e^{2s\varphi}dxdt \\
& \le k \int_0^T\int_{\alpha}^1 s \Theta (z_x)^2e^{2s\Phi} dxdt +k\int_0^T\int_{\alpha}^1 s^3\Theta^3 z^2e^{2s\Phi} dxdt\\
& \le C \int_0^T \int_{\omega} \frac{v^2}{a}dxdt,
\end{aligned}
\end{equation}
for a positive constant $C$ and $s$ large enough.
Thus \eqref{stimacar1} and \eqref{stimacar21N} imply again \eqref{carin01}.
As before we have to prove a similar inequality for
$x\in[0,\tau]$. We consider $W$ defined as in \eqref{WN}, but restriced to $[-1,1]$ and $\tilde a$, $\tilde b$, and $Z$ defined as for the case when (Dbc) holds. Then
$W$ satisfies the problem
\[
\begin{cases}
W_t +\tilde a W_{xx}+ \lambda \ds\frac{W}{\tilde b}= 0, &(t,x) \in  (0,T)\times (-1,1),\\
W_x(t,-1)=W_x(t,1) =0, & t \in (0,T),
\end{cases}
\]
while $Z$ satisfies
\eqref{eq-Z*} and \eqref{*}. As a consequence, if $v$ solves $(P_2)$, $v$ satisfies \eqref{car101}. Hence, the conclusion follows.

\end{proof}

\begin{proof}[Proof of Lemma $\ref{lemma3}$ if $x_0 \in \omega$]
$\quad$

{\it \underline{Dirichlet boundary conditions:}} Assume that  $v \in \cal Q$ solves
$(P_1)$. By assumption, we can find two subintervals
$\omega_1=(\lambda_1,\beta_1)\subset (0, x_0),
\omega_2=(\lambda_2,\beta_2) \subset (x_0,1)$ such that $(\omega_1
\cup \omega_2) \subset \subset \omega \setminus \{x_0\}$. Now, fix $\tilde \alpha \in (\alpha, \lambda_1)$, $\tilde \beta \in (\beta_2 , \beta)$ and consider a smooth function $\xi:[0,1]\to[0,1]$ such that
\[
\xi(x)=\begin{cases} 0,& x\in [0,\tilde\alpha],\\
1, & x\in [\lambda_1,\beta_2]\\
 0, &x\in [\tilde\beta,1],
\end{cases}
\]
and define $w:= \xi v$.
Hence, $w$ satisfies
\[
\begin{cases}
w_t + a  w_{xx}+ \lambda\ds \frac{w}{b} = a( \xi_{xx} v  + 2\xi _x  v_x) =:f,&
(t,x) \in Q_T, \\
w(t,0)= w(t,1)=0, & t \in (0,T).
\end{cases}
\]
Applying Theorem \ref{Cor1}, using the fact that $w_x(t,0)=w_x(t,1)=0$, the definition of $\xi$  and in
particular the fact that  $\xi_x$ and  $\xi_{xx}$ are supported
in $\tilde \omega:= [\tilde\alpha, \lambda_1] \cup[ \beta_2, \tilde\beta]\subset \subset \check{\omega}=[\alpha, \beta_1]\cup[\lambda_2, \beta]$,
we can write
\[
\begin{aligned}
&\int_{Q_T} \left(s\Theta (w_x)^2 + s^3 \Theta^3
\left(\frac{x-x_0}{a} \right)^2w^2\right)e^{2s\varphi}dxdt\\
& \le C\int_0^T \int_{\tilde\omega} (v^2+ v_x^2) e^{2s\varphi} dxdt
\le C \int_0^T \int_{\tilde \omega} \frac{v^2}{a} dxdt + C\int_0^T \int_{\tilde\omega} v_x^2 e^{2s\varphi} dxdt\\
&\le C \int_0^T \int_{\omega} \frac{v^2}{a} dxdt + C\int_0^T \int_{\tilde\omega} v_x^2 e^{2s\varphi} dxdt\\
&\text{(by Proposition \ref{caccio1})}\\
& \le C \int_0^T \int_{\omega} \frac{v^2}{a} dxdt + C\int_0^T \int_{\check\omega} \frac{v^2}{a} dxdt \le C \int_0^T \int_{\omega} \frac{v^2}{a} dxdt,
\end{aligned}
\]
for a positive constant $C$. Hence,
\begin{equation}\label{w}
\begin{aligned}
&\int_0^T \int_{\lambda_1}^{\beta_2}\left(s\Theta (v_x)^2 + s^3 \Theta^3
\left(\frac{x-x_0}{a} \right)^2v^2\right)e^{2s\varphi}dxdt
\\
&= \int_0^T \int_{\lambda_1}^{\beta_2} \left(s\Theta (w_x)^2 + s^3 \Theta^3
\left(\frac{x-x_0}{a} \right)^2w^2\right)e^{2s\varphi}dxdt\\
&\le \int_{Q_T} \left(s\Theta (w_x)^2 + s^3 \Theta^3
\left(\frac{x-x_0}{a} \right)^2w^2\right)e^{2s\varphi}dxdt \le C \int_0^T \int_{\omega} \frac{v^2}{a} dxdt.
\end{aligned}
\end{equation}
Now, consider the smooth function $\eta:[0,1]\to[0,1]$ such that
\begin{equation}\label{eta}
\eta(x)=\begin{cases} 1,& x\in [0,\lambda_1],\\
0, & x\in [\beta_1,1],
\end{cases}
\end{equation}
and define $z:= \eta v$; hence, $z$ satisfies
\[
\begin{cases}
z_t + a  z_{xx}+ \lambda\ds \frac{z}{b} = a( \eta _{xx}  z  + 2\eta _x  z_x) =:h,&
(t,x) \in(0, T)\times (0,1), \\
z(t,0)= z(t,1)=0, & t \in (0,T).
\end{cases}
\]
Applying Theorem \ref{Cor1}, using the fact that the boundary terms in \eqref{car} are non positive (observe that $z_x(t,1)=0$), and the fact that  $\eta_x$ and  $\eta_{xx}$ are supported
in $[\lambda_1, \beta_1]\subset \subset \hat \omega=[\tilde \alpha, \tilde \beta_1]$, where $\tilde \beta_1 \in (\beta_1, x_0)$ is fixed,
we get
\begin{equation}\label{z}
\begin{aligned}
&\int_0^T \int_{0}^{\lambda_1}\left(s\Theta (v_x)^2 + s^3 \Theta^3
\left(\frac{x-x_0}{a} \right)^2v^2\right)e^{2s\varphi}dxdt
\\
&= \int_0^T \int_{0}^{\lambda_1} \left(s\Theta (z_x)^2 + s^3 \Theta^3
\left(\frac{x-x_0}{a} \right)^2z^2\right)e^{2s\varphi}dxdt\\
&\le\int_{Q_T} \left(s\Theta (z_x)^2 + s^3 \Theta^3
\left(\frac{x-x_0}{a} \right)^2z^2\right)e^{2s\varphi}dxdt\\
& \text{(by Theorem \ref{Cor1})}\\
& \le C\int_0^T \int_{Q_T} \frac{h^2}{a} e^{2s\varphi} dxdt \le C\int_0^T \int_{\lambda_1}^{\beta_1} (v^2+ v_x^2) e^{2s\varphi} dxdt\\
&\le C\int_0^T \int_{\lambda_1}^{\beta_1} \left(\frac{v^2}{a}+ v_x^2\right) e^{2s\varphi} dxdt\\
&\le C \int_0^T \int_{\omega} \frac{v^2}{a} dxdt + C\int_0^T \int_{\lambda_1}^{\beta_1} v_x^2 e^{2s\varphi} dxdt\\
&\text{(by Proposition \ref{caccio1})}\\
& \le C \int_0^T \int_{\omega} \frac{v^2}{a} dxdt + C\int_0^T \int_{\hat\omega} \frac{v^2}{a} dxdt \le C \int_0^T \int_{\omega} \frac{v^2}{a} dxdt,
\end{aligned}
\end{equation}
for a positive constant $C$. Finally, consider the smooth function $\rho:[0,1]\to[0,1]$ such that
\begin{equation}\label{rho}
\rho(x)=\begin{cases} 1,& x\in [\beta_2, 1],\\
0, & x\in [0, \lambda_2],
\end{cases}
\end{equation}
and define $q:= \rho v$; hence, fixed $\tilde \lambda_2 \in(x_0, \lambda_2)$, $q$ satisfies
\[
\begin{cases}
q_t + a  q_{xx}+ \lambda\ds \frac{q}{b} = a( \rho _{xx} q  + 2\rho _x  q_x) =:H,&
(t,x) \in(0, T)\times (\tilde \lambda_2,1), \\
q(t,\tilde\lambda_2)= q(t,1)=0, & t \in (0,T).
\end{cases}
\]
The previous problem is non degenerate, so we can apply Proposition \ref{classical Carleman}. Since the boundary terms in \eqref{carcorretta} are non positive (observe that $q_x(t,\tilde \lambda_2)=0$) and $\rho_x$, $\rho_{xx}$ are supported
in $[\lambda_2, \beta_2]\subset \subset \breve{\omega}=[\tilde \lambda_2, \tilde \beta]$,
we get
\begin{equation}\label{q}
\begin{aligned}
&\int_0^T \int_{\beta_2}^{1}\left(s\Theta (v_x)^2 + s^3 \Theta^3
\left(\frac{x-x_0}{a} \right)^2v^2\right)e^{2s\varphi}dxdt
\\
&= \int_0^T \int_{\beta_2}^{1} \left(s\Theta (q_x)^2 + s^3 \Theta^3
\left(\frac{x-x_0}{a} \right)^2q^2\right)e^{2s\varphi}dxdt\\
&\le k\int_0^T \int_{\tilde \lambda_2}^{1} \left(s\Theta (q_x)^2 + s^3 \Theta^3
q^2\right)e^{2s\Phi}dxdt\\
& \le C\int_0^T \int_{\tilde \lambda_2}^{1}\frac{H^2}{a} e^{2s\Phi} dxdt \le C\int_0^T \int_{\lambda_2}^{\beta_2} (v^2+ v_x^2) e^{2s\varphi} dxdt\\
&\le C\int_0^T \int_{\lambda_2}^{\beta_2} \left(\frac{v^2}{a}+ v_x^2\right) e^{2s\varphi} dxdt\\
&\le C \int_0^T \int_{\omega} \frac{v^2}{a} dxdt + C\int_0^T \int_{\lambda_2}^{\beta_2} v_x^2 e^{2s\varphi} dxdt\\
&\text{(by Proposition \ref{caccio1} for the non degenerate case)}\\
& \le C \int_0^T \int_{\omega} \frac{v^2}{a} dxdt + C\int_0^T \int_{\breve\omega} \frac{v^2}{a} dxdt \le C \int_0^T \int_{\omega} \frac{v^2}{a} dxdt,
\end{aligned}
\end{equation}
for positive constants $k$ and $C$. Thus, by \eqref{w}, \eqref{z} and \eqref{q} the conclusion follows.

{\it\underline{Neumann boundary conditions:}} 
We proceed as for the Dirichlet case, obtaining \eqref{w}. Now, consider the cut-off function $\eta$ defined in \eqref{eta} and set $z:= \eta v$; hence, $z$ satisfies
\[
\begin{cases}
z_t + a  z_{xx}+ \lambda\ds \frac{z}{b} = a(\eta _{xx} z  + 2\eta _x  z_x) =:h,&
(t,x) \in Q_T, \\
z_x(t,0)= z_x(t,1)=0, & t \in (0,T).
\end{cases}
\]
Applying Theorem \ref{Cor1} and using the fact that  $\eta_x$ and  $\eta_{xx}$ are supported
in $[\lambda_1, \beta_1]\subset \subset \hat \omega=[\tilde \alpha, \tilde \beta_1]$, where $\tilde \beta_1$ is as before,
we get
\begin{equation}\label{zN}
\begin{aligned}
&\int_0^T \int_{0}^{\lambda_1}\left(s\Theta (v_x)^2 + s^3 \Theta^3
\left(\frac{x-x_0}{a} \right)^2v^2\right)e^{2s\varphi}dxdt
\\
&= \int_0^T \int_{0}^{\lambda_1} \left(s\Theta (z_x)^2 + s^3 \Theta^3
\left(\frac{x-x_0}{a} \right)^2z^2\right)e^{2s\varphi}dxdt\\
&\le\int_{Q_T} \left(s\Theta (z_x)^2 + s^3 \Theta^3
\left(\frac{x-x_0}{a} \right)^2z^2\right)e^{2s\varphi}dxdt\\
& \le C\left( \int_{Q_T} \frac{h^2}{a} e^{2s\varphi} dxdt + \int_0^T \int_{\omega} z^2 e^{2s\varphi}dxdt\right) \\
&\le C\int_0^T \int_{\lambda_1}^{\beta_1} (v^2+ v_x^2) e^{2s\varphi} dxdt + C\int_0^T \int_{\alpha}^{\beta_1} v^2e^{2s\varphi} dxdt\\
&\le C\int_0^T \int_{\lambda_1}^{\beta_1} \left(\frac{v^2}{a}+ v_x^2\right) e^{2s\varphi} dxdt+ C\int_0^T \int_{\omega}\frac{v^2}{a} dxdt\\
&\le C \int_0^T \int_{\omega} \frac{v^2}{a} dxdt + C\int_0^T \int_{\lambda_1}^{\beta_1} v_x^2 e^{2s\varphi} dxdt\\
&\text{(by Proposition \ref{caccio1})}\\
& \le C \int_0^T \int_{\omega} \frac{v^2}{a} dxdt + C\int_0^T \int_{\hat\omega} \frac{v^2}{a} dxdt \le C \int_0^T \int_{\omega} \frac{v^2}{a} dxdt,
\end{aligned}
\end{equation}
for a positive constant $C$. Finally, consider $q:= \rho v$, where $\rho$ is the cut-off function defined in \eqref{rho}; hence, fixed $\tilde \lambda_2$ as before, $q$ satisfies
\[
\begin{cases}
q_t + a  q_{xx}+ \lambda\ds \frac{q}{b} = a( \rho _{xx}  q  + 2\rho _x  q_x) =:H,&
(t,x) \in(0, T)\times (\tilde\lambda_2,1), \\
q_x(t,\tilde \lambda_2)= q_x(t,1)=0, & t \in (0,T).
\end{cases}
\]
The previous problem is non degenerate, so we can apply Proposition \ref{classical Carleman}. Since  $\rho_x$, $\rho_{xx}$ are supported
in $[\lambda_2, \beta_2]\subset \subset \breve{\omega}=[\tilde \lambda_2, \tilde \beta]$,
we get
\begin{equation}\label{qN}
\begin{aligned}
&\int_0^T \int_{\beta_2}^{1}\left(s\Theta (v_x)^2 + s^3 \Theta^3
\left(\frac{x-x_0}{a} \right)^2v^2\right)e^{2s\varphi}dxdt
\\
&= \int_0^T \int_{\beta_2}^{1} \left(s\Theta (q_x)^2 + s^3 \Theta^3
\left(\frac{x-x_0}{a} \right)^2q^2\right)e^{2s\varphi}dxdt\\
&\le k\int_0^T \int_{\tilde \lambda_2}^{1} \left(s\Theta (q_x)^2 + s^3 \Theta^3
q^2\right)e^{2s\Phi}dxdt\\
& \le C\left(\int_0^T \int_{\tilde \lambda_2}^{1}\frac{H^2}{a} e^{2s\Phi} dxdt + \int_0^T \int_{\lambda_2}^{\beta_2} q^2 e^{2s\Phi}dxdt\right)\le C\int_0^T \int_{\lambda_2}^{\beta_2} (v^2+ v_x^2) e^{2s\varphi} dxdt\\
&\le C\int_0^T \int_{\lambda_2}^{\beta_2} \left(\frac{v^2}{a}+ v_x^2\right) e^{2s\varphi} dxdt\\
&\le C \int_0^T \int_{\omega} \frac{v^2}{a} dxdt + C\int_0^T \int_{\lambda_2}^{\beta_2} v_x^2 e^{2s\varphi} dxdt\\
&\text{(by Proposition \ref{caccio1} for the non degenerate case)}\\
& \le C \int_0^T \int_{\omega} \frac{v^2}{a} dxdt + C\int_0^T \int_{\breve\omega} \frac{v^2}{a} dxdt \le C \int_0^T \int_{\omega} \frac{v^2}{a} dxdt,
\end{aligned}
\end{equation}
for positive constants $k$ and $C$. Thus, by \eqref{w}, \eqref{zN} and \eqref{qN} the conclusion follows.

\end{proof}

\begin{Lemma}\label{obser.regular}
Assume Hypotheses $\ref{ipotesiomega}$ - $\ref{ipa}$. There exists a
positive constant $C_T$ such that every solution $v \in \cal{Q}$ of
$(P_i)$, $i=1,2$, satisfies
\[
\int_0^1   \frac{1}{a}v^2(0,x) dx \le C_T\int_0^T \int_{\omega}\frac{1}{a}v^2(t,x)dxdt.
\]
  \end{Lemma}

  \begin{proof}
  Multiplying the equation of $(P_i)$, $i=1,2$, by $\ds\frac{v_t}{a}$ and integrating
  by parts over $(0,1)$, one has
\[
\begin{aligned}
&0 = \int_0^1\left(v_t+ av_{xx} +\lambda \displaystyle
\frac{v}{b}\right)\frac{v_t}{a} dx= \int_0^1 \left(\frac{1}{a}v_t^2+ v_{xx}v_t +\lambda
\displaystyle \frac{vv_t}{ab}\right)dx = \int_0^1\frac{1}{a}v_t^2dx +
\left[v_xv_t \right]_{x=0}^{x=1} \\&- \int_0^1v_xv_{tx} dx
+\frac{\lambda}{2}\frac{d}{dt}\int_0^1 \displaystyle
\frac{v^2}{ab}dx= \int_0^1\frac{1}{a}v_t^2dx -
\frac{1}{2}\frac{d}{dt}\int_0^1(v_x)^2
+\frac{\lambda}{2}\frac{d}{dt}\int_0^1 \displaystyle \frac{v^2}{ab}dx
 \\&\ge - \frac{1}{2}
\frac{d}{dt}\int_0^1 (v_x)^2dx
+\frac{\lambda}{2}\frac{d}{dt}\int_0^1 \displaystyle
\frac{v^2}{ab}dx.
\end{aligned}
\]
 Thus, the function $$t \mapsto
\displaystyle \int_0^1 (v_x)^2 dx - \lambda\int_0^1
\frac{v^2}{ab}dx$$ is non decreasing for all $t \in [0,T]$. In particular,
\begin{equation}\label{nondecrescente}
\begin{aligned}
\int_0^1 (v_x)^2(0,x)dx-\lambda\int_0^1
\frac{v^2(0,x)}{a(x)b(x)}dx & \le \int_0^1(v_x)^2(t,x)dx -\lambda\int_0^1
\frac{v^2(t,x)}{a(x)b(x)}dx.
\end{aligned}
\end{equation}
{\it\underline{If Dirichlet boundary conditions hold}}, then,
 by Lemma \ref{L2} or \ref{L2''},
\[
\begin{aligned}
&\int_0^1 (v_x)^2(0,x)dx-\lambda\int_0^1
\frac{v^2(0,x)}{a(x)b(x)}dx  \le \int_0^1(v_x)^2(t,x)dx -\lambda\int_0^1
\frac{v^2(t,x)}{a(x)b(x)}dx\\
& \le (1+ |\lambda| C^*)\int_0^1(v_x)^2(t,x)dx.
\end{aligned}
\]
Integrating the previous inequality
over $\displaystyle\left[\frac{T}{4}, \frac{3T}{4} \right]$, $\Theta$ being
bounded therein, we find
\begin{equation}\label{stima2}
\begin{aligned}
&\int_0^1(v_x)^2(0,x) dx - \lambda\int_0^1
\frac{v^2(0,x)}{a(x)b(x)}dx\le \frac{2}{T}(1+ |\lambda|
C^*)\int_{\frac{T}{4}}^{\frac{3T}{4}}\int_0^1(v_x)^2dxdt
\\&\le C_T
\int_{\frac{T}{4}}^{\frac{3T}{4}}\int_0^1s\Theta
(v_x)^2e^{2s\varphi}dxdt\\
&\text{(by Lemma \ref{lemma3})}\\
&\le C \int_0^T
\int_{\omega}\frac{v^2}{a}dxdt,
\end{aligned}
\end{equation}
for a strictly positive constant $C$.

Hence, from the previous inequality, if $\lambda \le 0$
\[
\int_0^1 (v_x)^2(0,x) dx \le \int_0^1(v_x)^2(0,x) dx -
\lambda\int_0^1 \frac{v^2(0,x)}{a(x)b(x)}dx\le C \int_0^T
\int_{\omega}\frac{v^2}{a}dxdt,
\]
for some positive constant $C>0$.

Now, suppose that $\lambda >0$. Then, by
\eqref{stima2}, one has
\[
\begin{aligned}
\int_0^1(v_x)^2(0,x) dx &\le \lambda\int_0^1
\frac{v^2(0,x)}{a(x)b(x)}dx + C \int_0^T \int_{\omega}\frac{v^2}{a}dxdt \\
&(\text{by Lemma \ref{L2} or \ref{L2''}})
\\
&\le
\lambda C^*\int_0^1(v_x)^2(0,x) dx + C \int_0^T
\int_{\omega}\frac{v^2}{a}dxdt.
\end{aligned}\]
Thus
\[
(1-\lambda C^*) \int_0^1(v_x)^2(0,x) dx \le C\int_0^T
\int_{\omega}\frac{v^2}{a}dxdt,
\]
 for a positive constant $C$. In every case, there exists $C >0$
 such that
 \begin{equation}\label{stum}
\int_0^1(v_x)^2(0,x) dx \le C\int_0^T \int_{\omega}\frac{v^2}{a}dxdt.
 \end{equation}
Now, 
 applying the Hardy--Poincar\'{e} inequality (see Proposition \ref{HP})
and \eqref{stum}, we have
\[
\begin{aligned}
\int_0^1 v^2(0,x)\frac{1}{a}dx &= \int_0^1 \frac{p(x)}{(x-x_0)^2}
v^2(0,x)dx\le \mathcal C_{HP} \int_0^1
p(x)(v_x)^2(0,x) dx \\&\le
\gamma\mathcal C_{HP} \int_0^1(v_x)^2(0,x) dx \le C
\int_0^T\int_{\omega}\frac{v^2}{a}dxdt,
\end{aligned}
\]
for a positive constant $C$.  Here $p(x) = \displaystyle
\frac{(x-x_0)^2}{a}$, $\mathcal C_{HP}$ is the Hardy--Poincar\'{e} constant
and $\gamma := \ds \max\left\{\frac{x_0^2}{a(0)}, \frac{(1-x_0)^2}{a(1)}\right\}$. Observe that the function $p$
satisfies the assumptions of Proposition \ref{HP} (with $q= 2-K_1$)
thanks to Lemma \ref{Lemma 2.1}. Hence,  the conclusion
follows.

{\it \underline{If Neumann boundary conditions hold.}} 
Assume, first of all, that $K_1+ K_2 <1$. Then, by \eqref{nondecrescente} and Lemma \ref{L2'}:
\begin{equation}\label{veditu}
\int_0^1 (v_x)^2(0,x)dx-\lambda\int_0^1
\frac{v^2(0,x)}{a(x)b(x)}dx
\le (1+ |\lambda|C^*)\left[\int_0^1(v_x)^2(t,x)dx+ \int_0^1v^2(t,x)dx\right].
\end{equation}
As before, integrating the previous inequality
over $\displaystyle\left[\frac{T}{4}, \frac{3T}{4} \right]$,  we find
\begin{equation}\label{stima2'}
\begin{aligned}
&\int_0^1(v_x)^2(0,x) dx - \lambda\int_0^1
\frac{v^2(0,x)}{a(x)b(x)}dx\\
&\le \frac{2}{T}(1+ |\lambda|
C^*)\int_{\frac{T}{4}}^{\frac{3T}{4}}\int_0^1((v_x)^2+ v^2)dxdt\\
&\le C_T
\int_{\frac{T}{4}}^{\frac{3T}{4}}\int_0^1s\Theta
(v_x)^2e^{2s\varphi}dxdt + C_T
\int_{\frac{T}{4}}^{\frac{3T}{4}}\int_0^1v^2dxdt \\
&\text{(by Lemma \ref{lemma3})}\\
&\le C \int_0^T
\int_{\omega}\frac{v^2}{a}dxdt+ C_T
\int_{\frac{T}{4}}^{\frac{3T}{4}}\int_0^1v^2dxdt,
\end{aligned}
\end{equation}
for a strictly positive constant $C$.

Now, we distinguish between the two cases $x_0 \in \omega$ and $x_0 \not \in \omega$.

If $x_0 \in \omega$, then
\[
\begin{aligned}
\int_{\frac{T}{4}}^{\frac{3T}{4}}\int_0^1v^2dxdt &\le \int_{\frac{T}{4}}^{\frac{3T}{4}}\int_{[0,1]\setminus \omega}v^2dxdt + \int_{\frac{T}{4}}^{\frac{3T}{4}}\int_\omega v^2dxdt\\
& \le C\left(\int_{\frac{T}{4}}^{\frac{3T}{4}}\int_{[0,1]\setminus \omega} s^3\Theta^3 \left( \frac{x-x_0}{a}\right)^2 v^2e^{2s\varphi}dxdt + \int_{\frac{T}{4}}^{\frac{3T}{4}}\int_\omega v^2dxdt\right)\\
& \le  C\left(\int_{Q_T} s^3\Theta^3 \left( \frac{x-x_0}{a}\right)^2 v^2e^{2s\varphi}dxdt + \int_0^T\int_\omega \frac{v^2}{a}dxdt\right)\\
&\text{(by Lemma \ref{lemma3})}\\
&\le C\int_0^T\int_\omega \frac{v^2}{a}dxdt.
\end{aligned}
\]
Substituting this inequality in \eqref{stima2'}, we obtain 
\begin{equation}\label{stima2''}
\int_0^1(v_x)^2(0,x) dx - \lambda\int_0^1
\frac{v^2(0,x)}{a(x)b(x)}dx\le C \int_0^T
\int_{\omega}\frac{v^2}{a}dxdt,
\end{equation}
where $C$ is a positive constant.

If $x_0 \not \in \omega$, then, by Corollary \ref{HPN1}, defining $p$ as before,
\[
\begin{aligned}
\int_{\frac{T}{4}}^{\frac{3T}{4}}\int_0^1v^2dxdt & \le \max_{[0,1]} a \int_{\frac{T}{4}}^{\frac{3T}{4}}\int_0^1\frac{p(x)}{(x-x_0)^2}v^2dxdt \\
& \le \max_{[0,1]} a\mathcal C_{HP}\int_{\frac{T}{4}}^{\frac{3T}{4}}\int_0^1(v_x^2+v^2)dxdt.
\end{aligned}
\]
Hence
\[
(1- \max_{[0,1]} a\mathcal C_{HP})\int_{\frac{T}{4}}^{\frac{3T}{4}}\int_0^1v^2dxdt \le \max_{[0,1]} a\mathcal C_{HP}\int_{\frac{T}{4}}^{\frac{3T}{4}}\int_0^1v_x^2dxdt.
\]
Since, by assumption, $\max_{[0,1]} a< \ds \frac{1}{C_{HP,1}}$ (indeed in this case $\mathcal C_{HP} = C_{HP,1}$), we have, using again Lemma \ref{lemma3},
\[
\begin{aligned}
&\int_{\frac{T}{4}}^{\frac{3T}{4}}\int_0^1v^2dxdt \le \frac{\max_{[0,1]} a\mathcal C_{HP}}{1- \max_{[0,1]} a\mathcal C_{HP}}\int_{\frac{T}{4}}^{\frac{3T}{4}}\int_0^1v_x^2dxdt\\
& \le C_T\int_{\frac{T}{4}}^{\frac{3T}{4}}\int_0^1s\Theta v_x^2e^{2s\varphi}dxdt\le C \int_0^T \int_\omega \frac{v^2}{a} dxdt.
\end{aligned}
\]
Again \eqref{stima2''} holds. In every case, under the given assumptions, one has \eqref{stima2''}.

Now, recall that, by assumption, if $K_1+ K_2<1$ one has that $\lambda <0$. Hence,  proceeding as for the (Dbc), one has
\begin{equation}\label{ultimobo}
\int_0^1 v_x^2(0,x) dx \le C \int_0^T \int_\omega \frac{v^2}{a}dxdt,
\end{equation}
for a positive constant $C$. Now, applying Corollary \ref{HPN1} and defining $p$ as before, it results
\[
\begin{aligned}
\int_0^1 \frac{v^2(0,x)}{a}dx&= \int_0^1 \frac{p(x)}{(x-x_0)^2} v^2(0,x) dx \le \mathcal C_{HP} \left [ \int_0^1 v^2(0,x) dx+ \int_0^1 v_x^2(0,x) dx\right]\\
& \le \mathcal C_{HP}\left[ \max_{[0,1]} a \int_0^1 \frac{v^2(0,x)}{a} dx+ \int_0^1 v_x^2(0,x) dx \right].
\end{aligned}
\]
Hence, by the previous inequality and \eqref{ultimobo},
\[
(1-\max_{[0,1]} a\mathcal C_{HP})  \int_0^1 \frac{v^2(0,x)}{a}dx \le C \int_0^T \int_\omega \frac{v^2}{a}dxdt.
\]
By assumption, the thesis follows.
\vspace{0.2cm} Assume now that one of Hypothesis \ref{Ass03_new}.2, \ref{Ass03_new}.3 or \ref{Ass03_new}.4 holds. Then, using Lemma \ref{L2'''}, \eqref{veditu} becomes
\[
\int_0^1 (v_x)^2(0,x)dx-\lambda\int_0^1
\frac{v^2(0,x)}{a(x)b(x)}dx
\le (1+ |\lambda|C^*)\int_0^1(v_x)^2(t,x)dx.
\]
Proceeding as for the case $K_1+K_2<1$, we can prove that \eqref{stima2''} holds and, if $\lambda \le 0$, the claim follows. Indeed in this case we have again \eqref{ultimobo} and, by Corollary \ref{HPN1} (using the fact that $v(x_0)=0$),
\[
\int_0^1 \frac{v^2(0,x)}{a}dx= \int_0^1 \frac{p(x)}{(x-x_0)^2} v^2(0,x) dx \le \mathcal C_{HP}  \int_0^1 v_x^2(0,x) dx \le C \int_0^T \int_\omega \frac{v^2}{a}dxdt,
\]
for a positive constant $C$. Again $p$ is as before.

On the other hand, if $\lambda >0$, by \eqref{2nondiv} and \eqref{stima2''}, we have
\[
\begin{aligned}
\int_0^1(v_x)^2(0,x) dx &\le \lambda\int_0^1
\frac{v^2(0,x)}{a(x)b(x)}dx + C \int_0^T \int_{\omega}\frac{v^2}{a}dxdt \\ &\le
\lambda C^* \int_0^1(v_x)^2(0,x) dx + C \int_0^T
\int_{\omega}\frac{v^2}{a}dxdt.
\end{aligned}\]
Thus
\[
(1-\lambda C^*) \int_0^1(v_x)^2(0,x) dx \le C\int_0^T
\int_{\omega}\frac{v^2}{a}dxdt,
\]
 for a positive constant $C$. By assumption $\lambda < \ds \frac{1}{C^*}$, hence there exists $C >0$
 such that
 \begin{equation}\label{stum1}
\int_0^1(v_x)^2(0,x) dx \le C\int_0^T \int_{\omega}\frac{v^2}{a}dxdt.
 \end{equation}
Since $v(x_0)=0$ by Lemma \ref{leso}, proceeding as before and using Corollary \ref{HPN1}, we get
\[
\begin{aligned}
\int_0^1 v^2(0,x)\frac{1}{a}dx &= \int_0^1 \frac{p(x)}{(x-x_0)^2}
v^2(0,x)dx\le \mathcal C_{HP} \int_0^1
(v_x)^2(0,x) dx \\
& \le C
\int_0^T\int_{\omega}\frac{v^2}{a}dxdt.
\end{aligned}
\]
Hence, also in this case, the conclusion
follows.

\end{proof}

The proof of Proposition $\ref{obser.}$ follows by a density argument as in
\cite[Proposition 4.1]{fm}.

\section*{Acknowledgments}

The author is a member of the Gruppo Nazionale per l'Analisi Matematica, la Probabilit\`a e le loro Applicazioni (GNAMPA) of the
Istituto Nazionale di Alta Matematica (INdAM) and she is partially supported  by the GDRE (Groupement De Recherche Europ\'een) CONEDP
(\emph{Control of PDEs}) and by the reaserch project {\em Sistemi con operatori irregolari}
of the GNAMPA-INdAM.

\end{document}